\documentclass[caot]{birkjour}
\usepackage{amsmath,amsthm,amsopn,amssymb,mathrsfs,times,newtxtext,newtxmath,upref}
\usepackage[mathscr]{eucal}

\usepackage{enumitem}
\setlist{label={$($\roman{enumi}\kern1pt$)$}}

\newtheorem{thm}{Theorem}[section]
\newtheorem{prop}[thm]{Proposition}
\newtheorem{cor}[thm]{Corollary}
\newtheorem*{cor*}{Corollary}
\newtheorem{lema}[thm]{Lemma}
\newtheorem*{lema*}{Lemma}

\numberwithin{equation}{section}
\theoremstyle{definition}
\newtheorem*{Def}{Definition}

\newtheorem{Example}{Example}
\newtheorem*{obs}{Remark}

\newcommand{\PI}[2]{\left\langle \,#1 , #2\, \right\rangle}

\newcommand{\St}{\mathcal{S}}
\newcommand{\HH}{\mathcal{H}}
\newcommand{\KK}{\mathcal{K}}
\newcommand{\M}{\mathcal{M}}
\newcommand{\N}{\mathcal{N}}
\newcommand{\mc}[1]{\mathcal{#1}}

\newcommand{\ol}{\overline}
\newcommand{\clran}{\ol{\mathrm{ran}}\,}

\newcommand{\LL}{{\mc{L}^{+\,2}}}
\newcommand{\LLcr}{{\mc{L}_{cr}^{+\,2}}}
\DeclareMathOperator{\ran}{ran}

\DeclareMathOperator{\Span}{span}
\DeclareMathOperator{\gm}{{\mathrm{\#}}}

\begin{document}

  \title{Products of positive operators}

  \author[Contino]{Maximiliano Contino}
	
  \address{
    Instituto Argentino de Matem\'atica ``Alberto P. Calder\'on'' \\
    CONICET\\
    Saavedra 15, Piso 3\\
    (1083) Buenos Aires, Argentina \\[5pt]
    Facultad de Ingenier\'{\i}a, Universidad de Buenos Aires\\
    Paseo Col\'on 850 \\
    (1063) Buenos Aires, Argentina}
    \email{mcontino@fi.uba.ar}
	
  \author[Dritschel]{Michael A.~Dritschel}
	
  \address{
  School of Mathematics, Statistics \& Physics \\
  Newcastle University\\
  Newcastle upon Tyne \\
  NE1 7RU, UK}
  \email{michael.dritschel@ncl.ac.uk}
	
  \author[Maestripieri]{Alejandra~Maestripieri}
	
  \address{
    Instituto Argentino de Matem\'atica ``Alberto P. Calder\'on'' \\
    CONICET\\
    Saavedra 15, Piso 3\\
    (1083) Buenos Aires, Argentina \\[5pt]
    Facultad de Ingenier\'{\i}a, Universidad de Buenos Aires\\
    Paseo Col\'on 850 \\
    (1063) Buenos Aires, Argentina}
    \email{amaestri@fi.uba.ar}
  
  \author[Marcantognini]{Stefania Marcantognini}
	
  \address{
    Instituto Argentino de Matem\'atica ``Alberto P. Calder\'on'' \\
    CONICET\\
    Saavedra 15, Piso 3\\
    (1083) Buenos Aires, Argentina \\[5pt]
    Universidad Nacional de General Sarmiento -- Instituto de Ciencias
    \\ Juan Mar\'ia Gutierrez 
    \\ (1613) Los Polvorines, Pcia. de
    Buenos Aires, Argentina}
    \email{smarcantognini@ungs.edu.ar}
	
    \keywords{products of positive operators; Schur complements;
      quasi-similarity; quasi-affinity; local spectral theory;
      generalized scalar operators}
    \subjclass{47A05, 47A65}

    \dedicatory{To Henk de Snoo, on his $75^{\mathrm{th}}$ birthday.}

  \date{\today}

  \begin{abstract}
    On finite dimensional spaces, it is apparent that an operator is
    the product of two positive operators if and only if it is similar
    to a positive operator.  Here, the class $\LL$ of bounded
    operators on separable infinite dimensional Hilbert spaces which
    can be written as the product of two bounded positive operators is
    studied.  The structure is much richer, and connects (but is not
    equivalent to) quasi-similarity and quasi-affinity to a positive
    operator.  The spectral properties of operators in $\LL$ are
    developed, and membership in $\LL$ among special classes,
    including algebraic and compact operators, is examined.
  \end{abstract}
    
  \maketitle
	
\section{Introduction}
\label{sec:introduction}

This work aims to shed light on two questions, ``Which bounded Hilbert
space operators are products of two bounded positive operators?'', and
``What properties do such operators share?''  Here, positive means
selfadjoint with non-negative spectrum.  This class is denoted
throughout by $\LL$.  The answer is easily given on finite dimensional
spaces: an operator will be in $\LL$ if and only if it is similar to a
positive operator~\cite{MR974046}, and this in turn is equivalent to
the operator being diagonalizable with positive spectrum.  Answering
the questions on infinite dimensional Hilbert spaces is a much more
delicate matter.  Similarity no longer suffices.

Apostol~\cite{MR402522} studied the question as to which operators are
quasi-similar to normal operators, and his work readily adapts to this
setting, making it possible to construct operators which are
quasi-similar to positive operators.  Another difficulty then arises,
since not every operator which is quasi-similar to a positive operator
will be the product of two \emph{bounded} positive operators.  For
this, something extra is needed.

This is not the end of the story though, since the quasi-similar
operators which are in $\LL$ only form a part of the whole class.  One
can relax the quasi-similarity condition to quasi-affinity.  Here
again, the class of operators which are in $\LL$ and which are
quasi-affine to a positive operator can be characterized.  However,
even this falls short of giving the entire class.  Nevertheless, it
comes close, and in general $T\in\LL$ has the property that it has
both a restriction and extension in $\LL$ which are quasi-affine to a
positive operator.

Despite the fact that similarity to a positive operator fails to
capture the whole of $\LL$, a surprising number of the spectral
properties of operators similar to positive operators do carry over.
It is an elementary observation that the spectrum of an operator in
$\LL$ is contained in $\mathbb R^+$, the non-negative reals.  Also, it
was observed by Wu~\cite{MR974046} that the only quasi-nilpotent
operator in the class is $0$.  It happens that operators which are
similar to positive operators are spectral operators, and so decompose
as the sum of a scalar operator (having a spectral decomposition) and
a quasi-nilpotent operator.  Moreover, in this case the
quasi-nilpotent part is~$0$.  Using local spectral theory, it is
possible to define an invariant linear manifold (so not necessarily
closed) on which an operator is quasi-nilpotent~\cite{MR1004421}.  In
case the operator has the single valued extension property, which
enables the definition of a unique local resolvent, this manifold is
closed.  Since the operators in $\LL$ have thin spectrum, they also
have the single valued extension property.  It then follows that for
any operator in $\LL$, the quasi-nilpotent part is the restriction to
the kernel.  In addition, for non-zero point spectra, there is no
non-trivial (generalized) Jordan structure.  These ideas enable the
study operators in $\LL$ which are either algebraic or compact.  While
the only operators in $\LL$ which are similar to positive operators
are scalar, all are \emph{generalized scalar operators} (having a
$C^\infty$ functional calculus).  Furthermore, the algebraic spectral
subspaces for operators in $\LL$ have the same form as that exhibited
by normal operators.

Elements of $\LL$ with closed range are the ones which behave most
similarly to the finite dimensional case, since they are similar to
positive operators.  In this case it is possible to explicitly
describe the Moore-Penrose inverse of the operator, and to find a
generalized inverse which is also in $\LL$.

A good deal of the paper hinges on a theorem due to
Sebesty\'en~\cite{MR739047}.  Sebesty\'en's theorem states that for
fixed operators $A$ and $T$, the equation $T = AX$ has a positive
solution if and only if $TT^* \leq \lambda AT^*$ for some
$\lambda > 0$.  A proof of a refined version is given in
Section~\ref{sec:preliminaries} using Schur complement techniques (see
also~\cite{MR3090434}), enabling $T\in \LL$ to be written as $AB$,
$A,B \geq 0$, where $\clran A = \clran T$ and $\clran B = \clran T^*$.
Such a pair $(A,B)$ is called \emph{optimal} for $T$.  Optimal pairs
happen to be extremely useful.

Section~\ref{sec:sim-quasi-sim-to-pos} looks at those operators (not
necessarily in $\LL$) which are either quasi-affine or quasi-similar
to positive operators.  Rigged Hilbert spaces are used to show that
for an operator $T$ quasi-affine to a positive operator,
$\ran T \cap \ker T = \{0\}$ and
$\overline{\ran T \dotplus \ker T} = \HH$.  In the quasi-similar case,
since this will hold for both $T$ and $T^*$, one has instead that
$\overline{\clran T \dotplus \ker T} = \HH$.  Work of Hassi,
Sebesty\'en, and de~Snoo~\cite{MR2174236} plays a key role in
describing those operators quasi-affine to a positive operator.

The paper then turns to describing general properties of the class
$\LL$ in Section~\ref{sec:the-set-l+2}.  Central here are optimal
pairs, the properties of which are explored in detail.  Examples are
given which show that operators in $\LL$ which are similar to a
positive operator, quasi-similar to a positive operator, and
quasi-affine to a positive operator form strictly increasingly larger
subclasses when $\dim \HH = \infty$, and that there are operators in
$\LL$ which do not fall into any of these, further hinting at the
complexities of the class.

Similarity to a positive operator completely characterizes $\LL$ on
finite dimensional spaces, and this is examined in
Section~\ref{sec:l+2-and-sim--crl+2}.  The closed range operators are
considered as a special sub-category.  In
Section~\ref{sec:l+2-and-quasi-sim}, attention turns to those
operators in $\LL$ which are either quasi-affine or quasi-similar to a
positive operator, where there are characterizations given which are
analogous to those found for operators similar to a positive operator.
Generally, there is only a weak connection between the spectra of
quasi-similar operators.  However, for an operator in $\LL$,
quasi-affinity to a positive operator preserves the spectrum.  It is
also proved that not every operator which is quasi-similar to a
positive operator is in $\LL$, by proving that any operator in $\LL$
which is quasi-similar to a positive operator has a square root which
is quasi-similar to a positive operator, but not all have square roots
in $\LL$.  In Section~\ref{sec:l+2--general-case}, general operators
in $\LL$ are considered, and the main point is that for any
$T \in \LL$, there exist both restrictions and extensions (on the same
Hilbert space) which are also in $\LL$ and which are quasi-affine to a
positive operator.

A constant refrain throughout is that operators in $\LL$ have many of
the properties of positive operators.
Section~\ref{sec:l+2-spectral-properties} examines this resemblance
with regards to local spectral properties.  This is applied in the
final section to algebraic operators and compact operators in $\LL$.

\section{Preliminaries}
\label{sec:preliminaries}

Throughout, all spaces are complex and separable Hilbert spaces.  The
domain, range, closure of the range, null space or kernel, spectrum and
resolvent of any given operator $A$ are denoted by
$\mathrm{dom}\,(A)$, $\ran A$, $\clran A$, $\ker A$, $\sigma(A)$, and
$\rho(A)$, respectively, and $\sigma(A) \subseteq [0,\infty)$ is
indicated as $\sigma(A) \geq 0$.

The space of everywhere defined bounded linear operators from $\HH$ to
$\KK$ is written as $L(\HH, \KK)$, or $L(\HH)$ when $\HH = \KK$, while
$CR(\HH)$ denotes the subset in $L(\HH)$ of closed range operators.
The identity operator on $\HH$ is written as $1$, or $1_{\HH}$ if it
is necessary to disambiguate.

As usual, the direct sum of two subspaces $\M$ and $\N$ with
$\M \cap \N= \{0\}$ is indicated by $\M \dotplus \N$, and the
orthogonal direct sum by $\M \oplus \N$.  The orthogonal complement of
a space $\mc{M}$ is written $\mc{M}^\perp$.  The symbol $\mc{P}$
denotes the class of all Hilbert space orthogonal projections, while
$P_{\M}$ is the orthogonal projection with range $\M$.

Write $GL(\HH)$ for the group of invertible operators in $L(\HH)$,
$\mc{L}^+ = L(\HH)^+$, the class of positive semidefinite operators,
$GL(\HH)^+ := GL(\HH) \cap \mc{L}^+$ and
$CR(\HH)^+ := CR(\HH) \cap \mc{L}^+$.  The paper focuses on the
operators in
\begin{equation*}
  \LL := \{T \in L(\HH): T = AB\text{ where }A,B\in\mc{L}^+\}.
\end{equation*}
Occasionally, this will be written as $\LL(\HH)$ if it is necessary to
clarify on which space the operators are acting.

Given two operators $S, T \in L(\HH)$, the notation $T \leq S$
signifies that $S-T \in \mc{L}^+$ (the \emph{L\"owner order}).  Given
any $T \in L(\HH)$, $\vert T \vert := (T^*T)^{1/2}$ is the modulus of
$T$ and $T = U\vert T\vert$ is the polar decomposition of $T$, with
$U$ the partial isometry such that $\ker U = \ker T$ and
$\ran U = \clran T$.

For $B \in \mc{L}^+$, the \emph{Schur complement} $B_{/ \St}$ of $B$
to a closed subspace $\St \subseteq \HH$ is the maximal element of
$\{ X \in L(\HH): \ 0\leq X\leq B$ and
$\ran X \subseteq \St^{\perp}\}$.  It always exists.  The
$\St$-\emph{compression} of $B$ is defined as
$B_{\St} := B- B_{/ \St}$.

Let $B \in L(\HH)$ be selfadjoint, $\St \subseteq \HH$ a closed
subspace, relative to $\St \oplus \St^\bot$,
\begin{equation*}
  B = \begin{pmatrix}
    B_{11}  & B_{12} \\
    B_{12}^* & B_{22} \\
  \end{pmatrix}.
\end{equation*}
Suppose that $B \geq 0$.  Write $B^{1/2} = \begin{pmatrix} R_1^* \\
  R_2^* \end{pmatrix}$, where $R_1^*$, $R_2^*$ are the rows of
$B^{1/2}$.  Then for $j= 1,2$, $R_j^*R_j = B_{jj}$, and so by Douglas'
lemma, there are isometries $V_j:\clran B_{jj}^{1/2} \to \clran R_j$
such that $R_j = V_j B_{jj}^{1/2}$.  Then
$B_{12} = R_1^* R_2 = B_{11}^{1/2} F B_{22}^{1/2}$, where
$F = V_1^* V_2 : \clran B_{22}^{1/2} \to \clran B_{11}^{1/2}$ is a
contraction.

On the other hand, if $B_{11},B_{22} \geq 0$ and $B_{12}$ has this
form, then
\begin{equation}
  \label{eq:1}
  \begin{split}
    B & =
    \begin{pmatrix}
      B_{11}^{1/2} & 0 \\
      B_{22}^{1/2} F^* & B_{22}^{1/2} D_F
    \end{pmatrix}
    \begin{pmatrix}
      B_{11}^{1/2} & F B_{22}^{1/2} \\
      0 &  D_F B_{22}^{1/2}
    \end{pmatrix} \\
    & =
    \begin{pmatrix}
      B_{11}^{1/2} \\
      B_{22}^{1/2} F^*
    \end{pmatrix}
    \begin{pmatrix}
      B_{11}^{1/2} & F B_{22}^{1/2}
    \end{pmatrix}
    +
    \begin{pmatrix}
      0 & 0 \\
      0 & B_{22}^{1/2} (1-F^*F) B_{22}^{1/2}
    \end{pmatrix},
  \end{split}
\end{equation}
where $D_F = (1_{\clran B_{22}} - F^*F)^{1/2}$ on $\clran B_{22}$.
Therefore, positivity of $B$ is equivalent to $B_{11},B_{22} \geq 0$
and the existence of such a contraction $F$.  The second term in the
sum in \eqref{eq:1} is the Schur complement $B_{/ \St}$, while the
first term is the $\St$-compression of $B$.  In general, it is not
difficult to verify that whenever $B = C^*C$, where
$C: \St \oplus \St^\bot \to \St$, then $B_{/ \St} = 0$.

The next theorem is a slightly strengthened form of one due to
Sebesty\'en (\cite{MR739047}, see also {\cite[Corollary
  2.4]{MR3090434}}).  It plays a central role in what follows.

\begin{thm}[Sebesty\'en]
  \label{Seb}
  Let $A, T \in L(\HH)$.  The equation $AX = T$ has a positive
  solution if and only if
  \begin{equation*}
    TT^* \leq \lambda A T^*
  \end{equation*}
  for some $\lambda \geq 0$, in which case $X$ can be chosen so that
  $\ker X = \ker T$, $X_{/\clran T} = 0$.  Furthermore, if $A \geq 0$
  with $\clran A = \clran T$, then $P_{\clran(T)}X|_{\clran(T)}$ will
  be injective with dense range.
\end{thm}

\begin{proof}
  If $AX = T$ has a positive solution, then $\lambda = \|X\|$
  suffices.  On the other hand, if for some $\lambda \geq 0$,
  $0 \leq TT^* \leq \lambda A T^*$, then $AT^* \geq 0$ and by Douglas'
  lemma, there exists $G$ with $\|G\| \leq \lambda^{1/2}$ and
  $\clran G \subseteq \clran (TA^*)^{1/2}$ satisfying
  $T = (TA^*)^{1/2}G$.  Clearly then, $\ker T = \ker G$ and
  $\clran T \subseteq \clran (TA^*)^{1/2}$.  Also
  $\clran(TA^*)^{1/2} = \clran(TA^*) \subseteq \clran T$, so equality
  holds.  The equality $TA^* = (TA^*)^{1/2} GA^*$ then implies
  $(TA^*)^{1/2} = GA^* = AG^*$.  Thus $T = AG^*G$, and so
  $X = G^*G \geq 0$ with $\ker X = \ker T$ (equivalently,
  $\clran X = \clran T^*$).  Also,
  $\clran G = \clran (TA^*)^{1/2} = \clran T$.  Decomposing
  $\HH = \clran T \oplus \ker T^*$, the operator $G$ has the form
  $G = \begin{pmatrix} G_1 & G_2 \\ 0 & 0 \end{pmatrix}$, and by
  \eqref{eq:1}, $X_{/\clran T} = 0$.

  Finally, if $A \geq 0$ with $\clran A = \clran T$, then $(TA)^{1/2}
  = GA = G_1A = AG_1^*$.  Hence $\ker(G_1^*G_1) = \{0\}$, and so
  $P_{\clran(T)}X|_{\clran(T)}$ is injective with dense range.
\end{proof}

\section{Similarity and quasi-similarity to a positive operator}
\label{sec:sim-quasi-sim-to-pos}

Recall that two operators $S, T \in L(\HH)$ are \emph{similar} if
there exists $G \in GL(\HH)$ such that $TG = GS$.

Mimicking the spectral theory for normal operators, an operator $T$ is
\emph{spectral} if for all $\omega \subseteq \mathbb C$ Borel, there
are (not necessarily orthogonal), uniformly bounded, countably
additive projections $E(\omega)$ commuting with $T$ such that
$\sigma(T|_{\ran{E(\omega)}}) = \overline{\omega}$.  If in addition,
$T= \int_{\sigma(T)} \lambda dE(\lambda)$, $T$ is termed a
\emph{scalar operator}, in which case it is similar to a normal
operator $A$ and $\sigma(T) = \sigma(A)$.  More generally, any
spectral operator $T$ has a unique decomposition $T = S + N$, where
$S$ is scalar, $N$ is quasi-nilpotent, and $SN = NS$.  See, for
example, \cite{MR63563}.

Various papers, including \cite[Theorem 2]{MR244801}, have considered
operators similar to selfadjoint operators.  See also \cite{MR190774}
for the connection with scalar operators.  The following collects
conditions for an operator to be similar to a positive operator.

\begin{thm}
  \label{thmLL}
  Let $T \in L(\HH)$.  The following statements are equivalent:
  \begin{enumerate}
  \item $TG = GS$ for some $G \in GL(\HH)$ and $S \in \mc{L}^+$;
  \item $TX = XT^*$ with $X \in GL(\HH)^+$ and $\sigma(T) \geq 0$;
  \item $T = AB$, with $A, B \in \mc{L}^+$, where $B$, respectively
    $A$, is invertible;
  \item There exist $W,Z \in GL(\HH)^+$ such that $TW \in \mc{L}^+$,
    respectively $ZT \in \mc{L}^+$;
  \item $T$ is a scalar operator and $\sigma (T) \geq 0$.
  \end{enumerate}
  If any of these hold, then
  \begin{equation*}
    \clran T \dotplus \ker T = \HH.
  \end{equation*}
\end{thm}

\begin{proof}
  $(\mathit{i}\kern0.5pt) \Rightarrow (\mathit{ii}\kern0.5pt)$: If
  $0 \leq S = G^{-1}TG$, $G \in GL(\HH)$, then
  $\sigma(T) = \sigma(S) \geq 0$.  Also, since $G^{-1}TG = G^* T^*
  G^{*\,-1}$, it follows that $(GG^*)^{-1}T(GG^*) = T^*$, or
  equivalently, $T(GG^*) = (GG^*)T^*$.

  $(\mathit{ii}\kern0.5pt) \Rightarrow (\mathit{iii}\kern0.5pt)$: Let
  $T = XT^*X^{-1}$, $X \in GL(\HH)^+$, and assume that
  $\sigma(T) \geq 0$.  Then
  \begin{equation*}
    X^{1/2}T^*X^{-1/2} = X^{-1/2}TX^{1/2} = (X^{-1/2}TX^{1/2})^* \in
    \mc{L}^+,
  \end{equation*}
  and so $A := X^{1/2}(X^{-1/2}TX^{1/2})X^{1/2} = TX \geq 0$.
  Consequently, $T = AB$, where $B = X^{-1} > 0$.  Work instead with
  $T^* = X^{-1\,}TX$ to obtain $T^* = BA$, $B \geq 0$ and $A > 0$.

  $(\mathit{iii}\kern0.5pt) \Rightarrow (\mathit{iv}\kern0.5pt)$:
  Suppose that $T = AB$, with $A, B \in \mc{L}^+$ and $B$ invertible.
  Let $W := B^{-1} \in GL(\HH)^+$.  Then $TW = A \in \mc{L}^+$.  If on
  the other hand $A$ is invertible, $Z = A^{-1}$ yields $ZT \geq 0$.

  $(\mathit{iv}\kern0.5pt) \Rightarrow (\mathit{i}\kern0.5pt)$:
  Suppose $W \in GL(\HH)^+$ and $TW \in \mc{L}^+$.  Then
  \begin{equation*}
    W^{-1/2} (TW) W^{-1/2} = W^{-1/2} T W^{1/2} \geq 0.
  \end{equation*}
  Similarly if $ZT \in \mc{L}^+$.
  
  $(\mathit{v}\kern0.5pt) \Leftrightarrow (\mathit{i}\kern0.5pt)$: If
  $G \in GL(\HH)$ is such that $S = G^{-1}TG \in \mc{L}^+$, then
  $\sigma(T) \geq 0$.  Let $E^S$ be the spectral measure of $S$, so
  that $S = \int_{\sigma(S)} \lambda \ dE^S(\lambda)$.  Then
  $E^T(\cdot) := G^{-1}E^S(\cdot)G$ is a resolution of the identity
  for $T$ and $T = \int_{\sigma(S)} \lambda \ dE^T(\lambda)$.  Thus
  $T$ is scalar.
	
  Conversely, if $T$ is scalar and $\sigma (T) \geq 0$, then $T$ is
  similar to $S$ normal with $\sigma(S) = \sigma (T) \geq 0$, and thus
  $S \in \mc{L}^+$.

  To prove the last statement, assume $(\mathit{i})$.  Since
  $S \geq 0$, $\clran S \oplus \ker S = \HH$.  Also,
  $\clran T = G\clran S$ and $\ker T = G \ker S$.  Since $G$ is
  injective, $\clran T \cap \ker T = \{0\}$, and since $G$ is
  surjective $\clran T + \ker T = \HH$.  Hence
  $\clran T \dotplus \ker T = \HH$.
\end{proof}

If $T\in \LL$ is similar to $S\geq 0$, with $TG = GS$ (where without
loss of generality in $(\mathit{i}\kern0.5pt)$ in the last theorem,
$G$ can be taken to be positive), then as previously noted,
$\Omega = \sigma(T) = \sigma(S) \subset \mathbb R^+$.  Moreover, since
it is possible to define a Borel functional calculus for $S$ on
$\Omega$, the same then holds for $T$ (see Theorem~\ref{thmLL}, where
this is essentially what is implied by $T$ being a scalar operator).
In particular, if $f$ is a Borel function, then $f(T) = Gf(S)G^{-1}$
is well-defined.

If $f(\Omega) \subset \mathbb R^+$, then $f(S) \geq 0$ and
\begin{equation*}
  f(T) = (Gf(S)G)(G^{-2}) \in \LL.
\end{equation*}
A case of particular interest is $f(x) = x^{1/2}$.  Since $T = AB$,
$A = (GSG)$, $B = (G^{-2})$, it follows that $T^{1/2} = A'B$ when
\begin{equation*}
  A'BA' = A, \qquad A' \geq 0.
\end{equation*}
This is an example of a \emph{Ricatti equation}, and more generally,
an operator $T\in \LL$ will have a square root if for some
factorization $T = AB$, $A,B \geq 0$, there exists $A' \geq 0$
satisfying this equality.  This is examined more closely later in the
section.

There is also a close connection with the \emph{geometric mean},
defined for two positive operators $E$ and $F$ with $E$ invertible as
\begin{equation*}
  E\gm F = E^{1/2}(E^{-1/2}FE^{-1/2})^{1/2}E^{1/2},
\end{equation*}
and so satisfies $(E\gm F)E^{-1}(E\gm F) = F$.

\begin{lema}
  \label{lem:sq-root-w-geom-mean}
  If $T$ is similar to a positive operator and $T = AB$ with $B\in
  GL(\HH)^+$, respectively, $A \in GL(\HH)^+$, then
  \begin{equation*}
    T^{1/2} = (B^{-1}\gm A) B,
  \end{equation*}
  respectively, $T^{1/2} = A(A^{-1} \gm B)$.
\end{lema}

\begin{proof}
  By Theorem~\ref{thmLL}, $B$ can be chosen invertible in $T = AB$,
  and then with $G = B^{-1/2}$ and $S = B^{1/2}AB^{1/2}$, $TG = GS$.
  Setting $E = B^{-1} = G^2$ and $F = A = GSG$, it follows that
  $E\gm F = GS^{1/2}G$, and hence $T^{1/2} = (B^{-1} \gm A) B$ since
  $(B^{-1} \gm A) B (B^{-1} \gm A) = A$.  If instead $A$ is chosen to
  be invertible, then working with $G^{-1}T = SG^{-1}$, one obtains
  $T^{1/2} = A(A^{-1} \gm B)$.
\end{proof}

An operator which is injective with dense range is termed a
\emph{quasi-affinity}.  An operator $T$ is \emph{quasi-affine} to $C$
if there is a quasi-affinity $X$ such that
\begin{equation*}
  TX = XC.
\end{equation*}
The operators $T$ and $C$ are said to be \emph{quasi-similar} if there
exist quasi-affinities $X, Y \in L(\HH)$ such that
\begin{equation*}
  TX = XC \quad\text{ and }\quad YT = CY.
\end{equation*}

A finite or countable system $\{\St_n\}_{ 1\leq n < m}$ of subspaces
of $\HH$ is called \emph{basic} if
$\St_n \dotplus \overline{\bigvee}_{k \neq n} \St_k = \HH$ for every
$n$ ($\overline{\bigvee}$ indicating the closed span), and
$\bigcap_{n \geq 1} (\overline{\bigvee}_{k \geq n} \St_k) = \{0\}$ if
$m = \infty$.  In \cite{MR402522}, Apostol uses basic systems to
characterize those operators which are quasi-similar to normal
operators.  With only minor modification, his proof works to
characterize quasi-similarity to positive operators.

\begin{thm}[Apostol]
  \label{Apostol}
  The operator $T \in L(\HH)$ is quasi-similar to a positive operator
  if and only if there exists a basic system $\{\St_n\}_{n \geq 1}$ of
  invariant subspaces of $T$ such that $T|_{\St_n}$ is similar to a
  positive operator.
\end{thm}

It is sometimes useful to relax the conditions in the definition of
quasi-similarity so that instead, $TX = XC$ and $YT = DY$.  The next
lemma shows that if $C$ and $D$ are positive, this is no more general.

\begin{lema}
  \label{lem:Propqs}
  Let $T \in L(\HH)$ such that $TX = XC$ and $YT = DY$, with $X, Y$
  quasi-affinities and $C, D \in \mc{L}^+$.  Then
  \begin{enumerate}
  \item $C$ is quasi-similar to $D$, and
  \item $T$ is quasi-similar to $C$.
  \end{enumerate}
\end{lema}
	
\begin{proof}
  $(\mathit{i}\kern0.5pt)$: Since $(YX)C = YTX = D(YX)$,
  $C(YX)^* = (YX)^*D$, and the claim follows since $YX$ and $(YX)^*$
  are quasi-affinities.

  $(\mathit{ii}\kern0.5pt)$: Set $Y' := (YX)^*Y$.  Then $Y'$ is a
  quasi-affinity and
  \begin{equation*}
    \begin{split}
      & Y'T = (YX)^*YT = (YX)^*DY = X^*Y^*DY \\
      =& X^*T^*Y^*Y = (TX)^*Y^*Y = CX^*Y^*Y = CY'.
    \end{split}
  \end{equation*}
  By assumption $TX = XC$, so it follows that $T$ is quasi-similar to
  $C$.
\end{proof}

\begin{Def}
  A \emph{rigged Hilbert space} is a triple $(\St, \HH, \St^*)$ with
  $\HH$ a Hilbert space and $\St \subseteq \HH$ a dense subspace such
  that the inclusion $\iota: \St \to \HH$ is continuous.  The space
  $\St^*$ is the dual of $\St$, and $\HH^* = \HH$ is mapped into
  $\St^*$ via the adjoint map $\iota^*$.  The spaces $\St$ and $\St^*$
  are identified as Hilbert spaces, with $\iota^*\iota(\HH) = \HH^*$.
  
  Let $X \in L(\HH)$.  Define an inner product on $\ran X$ by
  \begin{equation*}
    \PI{x}{y}_X := \PI{X^{-1}x}{X^{-1}y}, \qquad x,y \in \ran X.
  \end{equation*}
  Then $\HH_X := (\ran X,\PI{\cdot}{\cdot}_X)$ is a Hilbert space.
\end{Def}

The primary case of interest is when $X$ is a quasi-affinity, in which
case $\HH_X$ can be viewed as a rigged Hilbert space.

\begin{prop}
  \label{PropII}
  Let $T \in L(\HH)$ such that $TX = XC$ with $X$ a quasi-affinity and
  $C \in \mc{L}^+$.  Then $\HH_X$ can be identified with a rigged
  Hilbert space and $\hat{T} := T|_{\ran X} \in L(\HH_X)^{+}$.
  Furthermore, $\ran T \cap \ker T = \{0\}$ and
  \begin{equation*}
    \ol{\ran T \dotplus \ker T} = \HH.
  \end{equation*}
\end{prop}

\begin{proof}
  Let $y = XX^{-1}y \in \ran X$.  Then
  $\Vert y \Vert \leq \Vert X \Vert \Vert X^{-1} y \Vert = \Vert X
  \Vert \Vert y \Vert_X$.  Therefore the inclusion map
  $\iota: \HH_X \hookrightarrow \HH$ is continuous.  Thus $\HH_X$ (or
  more properly, the triple $(\HH_X, \HH, \HH_X^*)$) is a rigged
  Hilbert space.  This space is simply denoted as $\HH_X$.  Note that
  for any set $\St \subseteq \ran X$,
  $\ol{\St}^{\HH_X} \subseteq \ol{\St}$.

  Since $TX = XC$, $T(\ran X) \subseteq \ran X$ and $\hat{T}$ is well
  defined.  Also, if $y = Xx$,
  $v=Xw$ for some $x,w \in \HH$,
  \begin{equation*}
    \PI{\hat{T}y}{v}_{X} = \PI{X^{-1}Ty}{X^{-1}v} =
    \PI{X^{-1}TXx}{w} = \PI{X^{-1}XCx}{w} = \PI{Cx}{w}.
  \end{equation*}
  Since $\|y\|_X = \|x\|$ and $\|v\|_X = \|w\|$, taking the supremum
  over $y$ and $v$ with norm $1$ gives that
  $\Vert \hat{T}\Vert_X = \Vert C \Vert$ and so $\hat{T}$ is bounded
  in $\HH_X$.  Taking $v = y$ then yields $\hat{T} \in L(\HH_X)^{+}$.
  By positivity $\HH = \clran C \oplus\, \ker C$, and $\clran T =
  \overline{X\ran C}$, $\ker T = \overline{X\ker C}$, so
  $\clran^{\HH_X}(\hat{T}) \oplus_{\HH_X}\ker \hat{T} = \HH_X = \ran X
  \subseteq \clran T + \ker T$, and thus $\ol{\ran T+\ker T} = \HH$.

  Now suppose that $0 \neq z = Tx \in \ker T$.  There is a sequence
  $\{h_n\}$ in $\HH$ such that $XX^* h_n \to x$, and so
  $T^2 XX^*h_n = XC^2X^*h_n \to 0$.  Let $g_n = X^* h_n$ for all $n$.
  Then $X g_n \to x$ and for all $y\in \HH$,
  \begin{equation*}
    \PI{Cg_n}{CX^*y} = \PI{XC^2 X^* h_n}{y} \to 0.
  \end{equation*}
  Since $\clran(CX) = \clran C$, it follows that $Cg_n \to 0$.  Hence
  $TXX^* h_n = XCX^* h_n = XCg_n \to 0$, which implies that $Tx = 0$,
  a contradiction.
\end{proof}

\begin{cor}
  \label{CorII}
  Let $T \in L(\HH)$ be quasi-similar to a positive operator.  Then
  $\clran T \cap \ker T = \{0\}$ and $\clran T \dotplus \ker T$ is
  dense in $\HH$.
\end{cor}

\begin{proof}
  If $T \in L(\HH)$ is quasi-similar to a positive operator $C$, by
  Proposition~\ref{PropII}, $\ol{\ran T \dotplus \ker T} = \HH$ and
  $\ol{\ran T^* \dotplus \ker T^*} = \HH$.  Hence
  $\clran T \cap \ker T = \{0\}$, and so $\clran T \dotplus \ker T$ is
  dense in $\HH$.
\end{proof}

The following is a special case of more general results found
in~\cite[Corollary~2.12]{MR445319} and \cite[Theorem~2]{MR635584}.

\begin{lema}
  \label{lem:T-qa-pos-C-then-sp_T-contains-sp_C}
  If $T \in L(\HH)$ is quasi-affine to $C \in \mc{L}^+$, then
  $\sigma(T) \supseteq \sigma(C)$.
\end{lema}

If $T \in \LL$, then it will be shown that these spectra are equal
(Proposition~\ref{prop:t-qa-pos-then-spectra-equal}).

\begin{prop}
  \label{PropIa}
  Let $T \in L(\HH)$.  The following statements are equivalent:
  \begin{enumerate}
  \item $T$ is quasi-affine to a positive operator;
  \item $T^* = BA$, with $B$ a closed surjective positive operator and
    $A \in \mc{L}^+$;
  \item There exists a quasi-affinity $X \in \mc{L}^+$ such that
    $TX \in \mc{L}^+$.
  \end{enumerate}
\end{prop}

\begin{proof}
  $(\mathit{i}\kern0.5pt) \Rightarrow (\mathit{ii}\kern0.5pt)$:
  Assume $TG = GS$, $G$ a quasi-affinity, $S \geq 0$.  Then
  $GG^*T^* = GSG^*$ and
  \begin{equation*}
    T^* = (GG^*)^{-1} (GSG^*).
  \end{equation*}
  The operator $GG^*$ is a quasi-affinity, hence $(GG^*)^{-1}$ maps
  $\ran(GG^*)$ onto $\HH$, and it is thus surjective, closed, and so
  selfadjoint.  Since for all $x$,
  $\PI{(GG^*)^{-1}GG^*x}{GG^*x} = \PI{x}{GG^*x} \geq 0$, $(GG^*)^{-1}$
  is positive.

  $(\mathit{ii}\kern0.5pt) \Rightarrow (\mathit{iii}\kern0.5pt)$:
  Assume $(\mathit{ii}\kern0.5pt)$.  Since $B$ is surjective, by the
  closed graph theorem, $B^{-1}: \HH \to \mathrm{dom}\,(B)$ is
  bounded, and since $B$ is positive, $B^{-1}$ is a quasi-affinity,
  and is also positive.  Then $B^{-1}T^* = A \geq 0$, and the claim
  follows with X = $B^{-1}$.

  $(\mathit{iii}\kern0.5pt) \Rightarrow (\mathit{i}\kern0.5pt)$:
  Suppose there exists a quasi-affinity $X \in \mc{L}^+$ such that
  $TX = XT^* \geq 0$.  According to \cite[Theorem 5.1]{MR2174236}, if
  $A$, $B$, and $C$ are bounded operators with $A \geq 0$ and
  $AB = C^*A$, there exists a unique bounded $S$ with
  $\ker A \subseteq \ker S$ such that $A^{1/2}B = S A^{1/2}$ and
  $C^*A^{1/2} = A^{1/2} S$.  Translating to the present context, take
  $A = X$ and $B = C = T^*$.  Then there exists a bounded $S$ so that
  $X^{1/2}T^* = SX^{1/2}$, equivalently, $TX^{1/2} = X^{1/2}S^*$.
  Thus $S^* = X^{-1/2}TX^{1/2}$.

  For all $x \in \HH$ and $y = X^{1/2}x$,
  \begin{equation*}
    \PI{S^*y}{y} = \PI{X^{-1/2}TXx}{X^{1/2}x} = \PI{TXx}{x} \geq 0.
  \end{equation*}
  It follows by polarization that $S$ is selfadjoint, and so
  $S \geq 0$.
\end{proof}

\goodbreak
 
\begin{cor}
  \label{corpIa}
  Let $T \in L(\HH)$.  The following statements are equivalent:
  \begin{enumerate}
  \item $T$ is quasi-similar to a positive operator;
  \item $T = AB$, with $A$ a closed surjective positive operator and
    $B \in \mc{L}^+$, and $T^* = B'A'$, with $B'$ a closed surjective
    positive operator and $A' \in \mc{L}^+$;
  \item There exist quasi-affinities $W, Z \in \mc{L}^+$ such that
    $TW$ and $ZT \in \mc{L}^+$;
  \item There exists a basic system $\{\St_n\}_{n \geq 1}$ of
    invariant subspaces of $T$ such that for all $n$, $T|_{\St_n}$ is
    scalar and $\sigma(T|_{\St_n}) \geq 0$.
  \end{enumerate}
\end{cor}

\begin{proof}
  The equivalence of $(\mathit{i}\kern0.5pt)$ --
  $(\mathit{iii}\kern0.5pt)$ is a direct consequence of
  Proposition~\ref{PropIa}.  The last item is equivalent to
  $(\mathit{i}\kern0.5pt)$ by Theorem~\ref{Apostol}.
\end{proof}

The last result resembles Theorem~\ref{thmLL}, though under the weaker
condition of quasi-similarity it appears not to be possible to say
much about the spectrum of $T$ without some extra conditions.  See
Section~\ref{sec:l+2-and-quasi-sim}.

Coming back to square roots, suppose that $TG = GS$, where $G \geq 0$
is a quasi-affinity and $S \geq 0$.  Then there exists a densely
defined linear operator $R$ mapping $\ran X$ to itself such that
$RX = XC^{1/2}$.  However it may not be the case that $R$ is bounded.
Circumstances when it is will be addressed further on.

\section{The set $\LL$}
\label{sec:the-set-l+2}

The remainder of the paper is devoted to the study of the set of
products of two positive bounded operators,
\begin{equation*}
  \LL := \{ T \in L(\HH): T = AB \text{ with } A, B \in \mc{L}^+ \}.
\end{equation*}

The subclasses $\mc{P} \! \cdot \! \mc{P}$ and
$\mc{P} \! \cdot \! \mc{L}^+$ were considered in \cite{MR3090434} and
\cite{MR2775769}.

If $T \in \LL$, it is straightforward to check that $T^* \in \LL$ and
$GTG^{-1} \in \LL$ for all $G \in GL(\HH)$.  Then the \emph{similarity
  orbit} of $T$,
$\mathbb{O}_T := \{GTG^{-1}: G \in GL(\HH)\} \subseteq \LL$.  Also, it
can easily be verified that
$\{ T^n: n \in \mathbb{N}\} \subseteq \LL$.

From the basic fact that for two operators $C$ and $D$, $\sigma(CD)
\cup \{0\} = \sigma(DC) \cup \{0\}$, the following is immediate.

\begin{lema}
  \label{lemaT1a}
  Let $T = AB \in \LL$, $A,B \in \mc{L}^+$.  Then $\sigma(T) =
  \sigma(A^{1/2}BA^{1/2}) \geq 0$.
\end{lema}

\begin{proof}
  As already observed,
  $\sigma(T) \cup \{0\} = \sigma(A^{1/2}BA^{1/2}) \cup \{0\} \geq 0$.
  If $0 \notin \sigma(T)$, then $A$ and $B$ are invertible, so
  $0 \notin \sigma(A^{1/2}BA^{1/2})$.  Likewise,
  $0 \notin \sigma(A^{1/2}BA^{1/2})$ implies $0 \notin \sigma(T)$, and
  so the stated equality holds.
\end{proof}

\begin{Example}
  \label{exmpl:subnormals}
  Lemma,~\ref{lemaT1a} implies that a normal operator in $\LL$ is
  positive.  Suppose now that $T \in \LL$ is subnormal.  Let $N$ be
  the minimal normal extension of $T$.  Then
  $\sigma(N) = \sigma(T) \geq 0$, and so $N$ is positive.  Since $T$
  is the restriction of $N$ to an invariant subspace, it too is then
  positive.

  It will be proved in Proposition~\ref{propCR1} that an operator in
  $\LL$ with closed range is similar to a positive operator.  This
  will imply then that any partial isometry $V$ in $\LL$ is similar to
  an orthogonal projection, and so is itself a projection.  Since $V$
  is a contraction, this means that $\ran V$ is orthogonal to
  $\ker V$, and so $V \geq 0$ is an orthogonal projection.
\end{Example}

\begin{prop}
  \label{lemaT2}
  Let $T \in \LL$.  Then there exist $A, B \in \mc{L}^+$ such that
  $T = AB$, $\clran A = \clran T$ and $\ker B = \ker T$.  For this
  pair, $\ran B \cap \ker A = \ran A \cap \ker B = \{0\}$, and it
  follows then that
  \begin{equation*}
    \ran T \cap \ker T = \{0\}.
  \end{equation*}
\end{prop}

\begin{proof}
  Let $T = A_0B_0 \in \LL$.  Then, by Theorem~\ref{Seb}, there exists
  $B \in \mc{L}^+$ such that $T = A_0B$ and $\ker B = \ker T$.  On the
  other hand, $T^* = BA_0 \in \LL$ and again by Theorem~\ref{Seb},
  there exists $A \in \mc{L}^+$ such that $T^* = BA$ and
  $\ker A = \ker T^*$.  If $x \in \ran B \cap \ker A$ then $x = By$
  for some $y \in \HH$ and $0 = Ax = ABy = Ty$.  Hence
  $y \in \ker T = \ker B$, and so $x = 0$.  The other equality follows
  in a similar way.  It is then immediate from this that
  $\ran T \cap \ker T = \{0\}$.
\end{proof}

\begin{cor}
  \label{cor:range-maps-densely}
  If $T\in \LL$, then $\clran(T|_{\clran T}) = \clran T$.
\end{cor}

\begin{proof}
  If $T \in \LL$, then $T = AB$ with $\clran T = \clran A$ by
  Proposition~\ref{lemaT2}.  Therefore,
  $\clran(T|_{\clran T}) = \clran(T P_{\clran T}) = (\ker (P_{\clran
    T} T^*))^\bot$.  But
  $\ker (P_{\clran T} T^*) = (\ran T^* \cap \ker T^*) + \ker T^* =
  \ker T^*$, again by Proposition~\ref{lemaT2}.
\end{proof}

\begin{Def}
  For $A, B \in \mc{L}^+$, the pair $(A,B)$ is called \emph{optimal}
  for $T = AB$ if $\clran T = \clran A$ and $\ker B = \ker T$.
\end{Def}

According to Proposition~\ref{lemaT2}, whenever $T \in \LL$, it can be
written as a product involving an optimal pair.  Clearly, the pair
$(A,B)$ is optimal for $T$ if and only if the pair $(B,A)$ is optimal
for $T^*$.

\begin{Example}
  \label{exmpl:oblique-proj-opt-pair}
  Any oblique projection $Q$ is in $\LL$.  For suppose that
  $\M = \ran Q$.  Then $QP_\M = P_\M = P_\M Q^*$ and $P_\M Q = Q$.
  Therefore $Q = P_\M (Q^*Q)$.  Obviously, $(P_\M,Q^*Q)$ is an optimal
  pair for $Q$.
\end{Example}

If $T \in \mc{P}^2$ then $\clran T \cap \ker T = \{0\}$ (see
\cite[Theorem 3.2]{MR2775769}).  This no longer need be the case if
$T \in \LL$, as the following example shows.

\begin{Example}{\cite[Lemma 3.1]{MR3090434}}
  \label{example-1}
  Let $A \in \mc{L}^+$ with non-closed dense range and
  $x \in \clran A\setminus \ran A$.  Define $\St = \Span\{x\}^{\perp}$
  and $T = AP_{\St} \in \LL$.  Then $\ker T = \Span\{x\}$,
  $\clran T^* = \St$,
  $\ker T^* = \{y : Ay \in \ker P_{\St}\} = \{0\}$, and
  $\clran T = \HH$.  Hence $\clran T \cap \ker T = \Span\{x\}$,
  $\clran T^* \cap \ker T^* = \{0\}$, and
  $\ol{\clran T^* \dotplus \ker T^*} = \St$.  By Proposition 3.4,
  $T^*$ is not quasi-affine to any positive operator,
  though $T$ is.

  If instead $T = (A \oplus B)(B\oplus A)$ on $\HH \oplus \HH$, then
  for neither $T$ nor $T^*$ is the closure of the range intersected
  with the kernel nontrivial, nor the sum of the range with the kernel
  dense in $\HH \oplus \HH$.  As a consequence of Proposition 3.4,
  neither is quasi-affine to a positive operator.  Clearly, in these
  examples $T$ is not quasi-similar to a positive operator.
\end{Example}

Operators $T\in \LL$ with a factorization $T = AB$ where one of $A$ or
$B$ has closed range have special properties (see, for example,
Theorem~\ref{prop:T-in-L+2-gen-scalar} and
Corollary~\ref{cor:L+2-w-A-or-B-closed-range}).

\begin{prop}
  \label{prop:T-in-L+2-w-A-or-B-closed-range}
  Let $T \in \LL$.  If $T$ is similar to a positive operator, then
  there exists an optimal pair with $\ran A$, respectively $\ran B$,
  closed.
\end{prop}

\begin{proof}
  Suppose that $T \in \LL$ is similar to a positive operator, so
  $T = GCG^{-1}$, $C \geq 0$.  Let $P$ be the projection onto the
  closure of the range of $C$.  Then $GPG^*$ and $G^{*\,-1}PG^{-1}$
  have closed range, and $T = (GPG^*)(G^{*\,-1}CG^{-1}) =
  (GCG^*)(G^{*\,-1}PG^{-1})$.  It is readily seen that
  $(GPG^*,G^{*\,-1}CG^{-1})$ and $(GCG^*,G^{*\,-1}PG^{-1})$ are
  optimal pairs.
\end{proof}

It is natural to wonder at this point if the class of operators in
$\LL$ which are quasi-similar to a positive operator is strictly
larger than the class of those which are similar to a positive
operator.

\begin{Example}
  \label{exmpl:qs-to-pos-but-no-A-or-B-w-clsd-rng}
  As noted, if $T \in \mc{P}^2$ then $\clran T \cap \ker T = \{0\}$.
  Furthermore, $\clran T \dotplus \ker T = \HH$ if and only if
  $\ran T$ is closed.  An operator $T\in \mc{P}^2$ without closed
  range is constructed as follows.  Assuming $\dim \HH = \infty$,
  there exist two closed subspaces $\mc{M}$ and $\mc{N}$ such that
  $\mc{M} \cap \mc{N} = \{0\}$ and $\mc{M} \dotplus \mc{N}$ is dense
  in, but not equal to $\HH$.  Take $T = AB$, $A$ and $B$ be
  orthogonal projections onto $\mc{M}$ and $\mc{N}^\bot$,
  respectively.  Then $\clran T = \mc{M}$ and $\ker T = \mc{N}$, so
  $\clran T \dotplus \ker T$ is dense in, but not equal to $\HH$.

  Let $W = A + P_{\mc{N}}$ and $Z = B + P_{\mc{M}^\bot}$.  Clearly,
  both are positive.  Also,
  $\ker W = \mc{M}^\bot \cap \mc{N}^\bot = \{0\}$, and similarly,
  $\ker Z = \{0\}$, so both are quasi-affinities.  Since $TW = ABA$
  and $ZT = BAB$ are both positive, it follows from
  Corollary~\ref{corpIa} that $T$ is quasi-similar to a positive
  operator.  By Theorem~\ref{thmLL}, $T$ cannot be similar to a
  positive operator, since $\clran T \dotplus \ker T \neq \HH$.
  
  The above example can also be used to construct $T\in \LL$ which
  again is quasi-similar, but not similar to a positive operator, but
  now with $\ker T = \ker T^* = \{0\}$.  Let $\HH = \KK \oplus \KK$,
  where $\dim \KK = \infty$.  Define $\mc{M} := \KK \oplus \{0\}$, and
  choose $\mc{N}$ as above.  Notice that $\dim \mc{N} = \dim \mc{M}$,
  so there is a unitary $V$ on $\HH$ mapping $\mc{N}$ to $\mc{M}$ and
  $\mc{N}^\bot$ to $\mc{M}^\bot$.

  Let $A_1 = P_{\mc{M}}$, $B_1 = P_{\mc{N}^\bot}$, $A_2 = P_{\mc{N}}$,
  $B_2 = P_{\mc{M}^\bot}$.  So $A_2 = V^*A_1V$ and $B_2 = V^*B_1V$.
  Set $A = \tfrac{1}{\sqrt{2}}(A_1 + A_2)$,
  $B = \tfrac{1}{\sqrt{2}}(B_1 + B_2)$.  These are both positive and
  injective, but since neither $\mc{M} \dotplus \mc{N}$ nor
  $\mc{N}^\bot \dotplus \mc{M}^\bot$ equals $\HH$, the ranges of $A$
  and $B$ are not closed.

  Let $W = \tfrac{1}{\sqrt{2}} \begin{pmatrix} 1 \\ V \end{pmatrix}$,
  and set
  \begin{equation*}
    T = AB = A_1B_1 + A_2B_2 = W^* (A_1B_1 \otimes 1_2) W,
  \end{equation*}
  where $A_1B_1 \otimes 1_2$ is the $2\times 2$ diagonal operator
  matrix with diagonal entries $A_1B_1$.  The operator $W$ is an
  isometry, and $T$ is injective with dense range.

  Suppose that $T$ is similar to a positive operator, $T = GCG^{-1}$.
  The operators
  \begin{equation*}
    W' := \frac{1}{\sqrt{2}}
    \begin{pmatrix}
      -V^* \\ 1
    \end{pmatrix},
    \qquad
    W'' := \frac{1}{\sqrt{2}}
    \begin{pmatrix}
      V^* \\ 1
    \end{pmatrix}    
  \end{equation*}
  are also isometric and $U = \begin{pmatrix} W & W' \end{pmatrix}$ is
  unitary.  Furthermore,
  \begin{equation*}
    \begin{split}
      & {W'}^*(A_1B_1 \otimes 1_2)W'
      = {W''}^*
      \begin{pmatrix}
        -1 & 0 \\ 0 & 1
      \end{pmatrix}
      (A_1B_1 \otimes 1_2)
      \begin{pmatrix}
        -1 & 0 \\ 0 & 1
      \end{pmatrix}
      W'' \\
      = & {W''}^*
      (A_1B_1 \otimes 1_2)
      W''
      = VW^*(A_1B_1 \otimes 1_2)WV^*
      = (VG)C(VG)^{-1}.
    \end{split}
  \end{equation*}
  Hence
  \begin{equation*}
    A_1B_1 \otimes 1_2 = U
    \begin{pmatrix}
      G & 0 \\ 0 & VG
    \end{pmatrix}
    (C \otimes 1_2)
    {\begin{pmatrix}
        G & 0 \\ 0 & VG
      \end{pmatrix}}^{-1}
    U^{-1};
  \end{equation*}
  that is, $A_1B_1 \otimes 1_2$ is similar to a positive operator.
  But by the same reasoning employed in showing that $A_1B_1$ is
  quasi-similar, but not similar to a positive operator, the same
  holds for $A_1B_1 \otimes 1_2$, giving a contradiction.  Hence, $T$
  is also quasi-similar, but not similar to a positive operator.

  It is noteworthy that if the operator $T$ just constructed were to
  have a factorization $T = AB$, where one of $A$ or $B$ has closed
  range, then by Theorem~\ref{thmLL}, $T$ would be similar to a
  positive operator.  Hence there can be no such factorization for
  this $T$.
\end{Example}

The following characterization of the elements of $\LL$ is immediate
from Theorem~\ref{Seb}.

\begin{thm}
  \label{thm:T-in-LL=pos-soln}
  Let $T \in L(\HH)$.  Then $T \in \LL$ if and only if the inequality
  $TT^* \leq X T^*$ admits a positive solution.
\end{thm}

\begin{proof}
  If $T \in \LL$ then there exist $A,B \in \mc{L}^+$ such that
  $T = AB$.  Since $B^2 \leq \Vert B \Vert B$ then
  $TT^* = AB^2A \leq \Vert B \Vert ABA = \Vert B \Vert AT^*$.
  Therefore, $\Vert B \Vert A$ is a positive solution of
  $TT^* \leq X T^*$.  Conversely, if $A \in \mc{L}^+$ satisfies
  $TT^* \leq AT^*$ then, by Theorem~\ref{Seb}, the equation $T = AX$
  admits a positive solution.  Therefore $T \in \LL$.
\end{proof}

\begin{cor}
  \label{cor:T=AB-iff-A=pos-soln}
  Let $T \in \LL$ and $A \in \mc{L}^+$.  Then $T$ can be factored as
  $T = AB$, with $B \in \mc{L}^+$ if and only if $\lambda A$ is a
  solution of $TT^* \leq X T^*$, for some $\lambda\geq 0$.
\end{cor}

\begin{cor}
  \label{cor:TTstar=ATstar}
  The operator $T\in \mc{L}^+ \! \cdot \! \mc{P}$ if and only if
  $TT^* = X T^*$ admits a positive solution.  Moreover,
  $T \in \mc{P}^2$ if and only if $TT^* = X T^*$ admits a solution in
  $\mc{P}$.
\end{cor}

\begin{proof}
  If $T = AP$, $A\geq 0$ and $P$ an orthogonal projection, then
  $TT^* = AP^2A = APA = AT^*$.  Conversely, if $TT^* = X T^*$ admits a
  positive solution $X = A \geq 0$, then $|T^*|^2 = AU|T^*|$, where
  $U$ is a partial isometry from $\clran T$ onto $\clran T^*$.  Thus
  $|T^*| = AU = U^*A$, and so $T^* = UU^*A$, where $UU^*$ is an
  orthogonal projection.
\end{proof}

The next result will be particularly useful for describing spectral
properties of elements of $\LL$ in
Section~\ref{sec:l+2-spectral-properties}.  It was proved for
invariant subspaces in the finite dimensional case in~\cite{MR974046}.
Recall that a subspace $\M$ is \emph{invariant} for an operator $T$ if
$T\M \subseteq \M$.

\begin{prop}
  \label{prop:L+2-restricted-to-inv-subsp-in-L+2}
  Let $T \in \LL$ and suppose $\M$ is invariant for $T$.  Then
  $TP_\M \in \LL$.
\end{prop}

\begin{proof}
  Write $T = AB$, $A,B \in \mc{L}^+$.  Then $T^*T \leq \lambda BT$ for
  $\lambda = \|A\|$.  Assume that $\M$ is invariant.  Then
  \begin{equation*}
    P_\M T^* T P_\M \leq \lambda P_\M B T P_\M = \lambda P_\M B P_\M T
    P_\M.
  \end{equation*}
  Since $\lambda P_\M B P_\M \geq 0$, by
  Theorem~\ref{thm:T-in-LL=pos-soln}, $TP_\M \in \LL$.
\end{proof}

From the proof of Theorem~\ref{thm:T-in-LL=pos-soln}, $T P_\M$ above
has the form $C (P_\M BP_\M)$ for some $C \in \mc{L}^+$.

In fact, it is not difficult to see that since $T^* \in \LL$, the
above proposition is true more generally for \emph{semi-invariant}
subspaces; that is, subspaces of the form $\M = \M_1 \ominus \M_2$,
where $\M_1$ and $\M_2$ are invariant for $T$.

\begin{Def}
  Given $T \in \LL$ and $A \geq 0$ with $\clran A = \clran T$, define
  \begin{equation*}
    \mc{B}_T^A = \{X \geq 0: T = AX\}.
  \end{equation*}
\end{Def}

Note that even if $\clran A = \clran T$ and $\ran A \supseteq \ran T$,
the set $\mc{B}_T^A$ may be empty.  As just seen in
Corollary~\ref{cor:T=AB-iff-A=pos-soln}, $A$ must also satisfy
$TT^* \leq \lambda AT^*$ for some $\lambda > 0$.

\begin{thm}
  \label{propopt}
  Let $T \in \LL$ and $A \geq 0$ be such that
  $\mc{B}_T^A \neq \emptyset$.  Then $\mc{B}_T^A$ has a minimum $B_0$.
  The pair $(A,B_0)$ is optimal and the set $\mc{B}_T^A$ is the cone
  \begin{equation*}
    \mc{B}_T^A = \{B_0 + Z: Z \in \mc{L}^+ \text{ and }\clran Z
    \subseteq \ker T^*\}.
  \end{equation*}
  Moreover, for every $B \in \mc{B}_T^A$, $B_{\clran T} = B_0$, and the
  pair $(A,B)$ is optimal if and only if
  $\ran Z \subseteq \clran T^* \cap \ker T^*$.
\end{thm}

\begin{proof}
  Let $B \in \mc{B}_T^A$ and $B_0 = G^*G$ the solution of $T = AX$
  constructed in the proof of Theorem~\ref{Seb}.  With respect to the
  decomposition $\HH = \clran T \oplus \ker T^*$, $B$ has an
  LU-decomposition,
  \begin{equation*}
    B = F^*F =
    \begin{pmatrix}
      F_1^* & 0 \\ F_2^* & F_3^*
    \end{pmatrix}
    \begin{pmatrix}
      F_1 & F_2 \\ 0 & F_3
    \end{pmatrix}.
  \end{equation*}
  Also by Theorem~\ref{Seb},
  \begin{equation*}
    B_0 = \begin{pmatrix} G_1^* & 0 \\ G_2^* & 0 \end{pmatrix}
    \begin{pmatrix} G_1 & G_2 \\ 0 & 0 \end{pmatrix}.
  \end{equation*}
  Since the theorem also gives in this circumstance that $G_1^*G_1$ is
  a quasi-affinity, there is no loss in generality in taking
  $G_1 \geq 0$ with dense range.

  Now
  \begin{equation*}
    TA = ABA = AF_1^*F_1 A = AB_0A = AG_1^2A,
  \end{equation*}
  and since $\clran A = \clran T$, $F_1^*F_1 = G_1^2$.  Without loss
  of generality, take $F_1 = G_1$ (adjusting $F_2$ and $F_3$ as
  necessary).  So
  \begin{equation*}
    T = AG^*G = A
    \begin{pmatrix}G_1^2 & G_1G_2 \end{pmatrix}
    = AF^*F =
    A \begin{pmatrix}G_1^2 & G_1F_2 \end{pmatrix}.
  \end{equation*}
  Therefore,
  \begin{equation*}
    G_2^*G_1A = F_2^*G_1A.
  \end{equation*}
  Since both $G_1$ and $A$ are positive with dense ranges in
  $\clran T$, $\clran (G_1A) = \clran T$.  Hence by continuity,
  $F_2 = G_2$.  Therefore
  $F = \begin{pmatrix} G_1 & G_2 \\ 0 & F_3 \end{pmatrix}$ and
  \begin{equation*}
    Z := B_{/\clran T} =
    \begin{pmatrix}
      0 & 0 \\ 0 & F_3^*F_3
    \end{pmatrix} \geq 0,
  \end{equation*}
  giving $B = B_0 + Z$, $Z \geq 0$, $\clran Z \subseteq \ker T^*$, and
  $B_{\clran T} = B_0$.

  Finally, suppose that the pair $(A,B)$ is optimal.  So
  $\clran B = \clran T^*$, where $B = B_0 +Z$, and since
  $\clran B_0 = \clran T^*$, it must be that
  $\clran Z \subseteq \clran T^*$.  Hence,
  $\clran Z \subseteq \clran T^* \cap \ker T^*$.  On the other hand,
  if $B = B_0 + Z$, $Z \geq 0$, and
  $\clran Z \subseteq \clran T^* \cap \ker T^*$, then
  $\clran B = \clran T^*$, and so $(A,B)$ is optimal.
\end{proof}

Theorem~\ref{propopt} states that if $T \in \LL$ admits an optimal
pair $(A,B)$, then there is an optimal pair $(A,B_0)$ where $B_0$ has
minimal norm among the operators in the set $\mc{B}_T^A$.
Furthermore, $B_0$ is the minimal positive completion of the operator
matrix
$\begin{pmatrix} B_{11} & B_{12} \\ B_{12}^* & * \end{pmatrix}$.
However, $(A,B_0)$ need not be the unique optimal pair for $T$ with
$A$ as the first factor.

\begin{Example}
  \label{example-2}
  Consider $T = P_\St A$, where $A \geq 0$ and $\St$ are defined as in
  Example~\ref{example-1}.  Then by Theorem~\ref{propopt}, for any
  $\lambda > 0$, $(P_\St, A+\lambda (1-P_\St))$ is an optimal pair for
  $T$.
\end{Example}

The next result gives a condition for the optimal pair $(A,B)$ to be
unique when one of the terms is fixed.

\begin{cor}
  \label{coruniqueBTA}
  Let $T \in \LL$ and $A$ such that $\mc{B}_T^A \neq \emptyset$ with
  minimal element $B_0$.  Then $(A,B_0)$ is the unique optimal pair
  for $T$ with $A$ as the first factor if and only if
  $\ol{\ran T \dotplus \ker T} = \HH$.  Additionally, for fixed $A$,
  $(A,B_0)$ and $(B_0,A)$ are unique optimal pairs for $T$ and $T^*$,
  respectively, if and only if $\ol{\clran T \dotplus \ker T} = \HH$.
\end{cor}

\begin{proof}
  This follows directly from Theorem~\ref{propopt}, since there can be
  more than one optimal pair $(A,B)$ for fixed $A$ if and only if
  $\clran T^* \cap \ker T^* \neq \{0\}$.  The condition $\ol{\clran T
    \dotplus \ker T} = \HH$ implies that $\clran T^* \cap \ker T^* =
  \{0\}$, and by a similar argument as at the end of the proof of
  Theorem~\ref{propopt}, this condition is necessary and sufficient
  for there to be a unique optimal pair $(B,A)$ for $T^*$ with $A$
  fixed.
\end{proof}

There is a dilation theory for elements of $\LL$ which mimics that of
contractions on Hilbert spaces.

\begin{prop}
  \label{prop:dilations-for-L+2}
  Let $T \in \LL$.  Then there is a Hilbert space
  $\HH' \supseteq \HH$, and an operator
  $T' \in \mc{L}^+ \cdot \mc{P}$ on $\HH'$ such that $T$ is the
  restriction of $T'$ to an invariant subspace,
  $\clran T' = \clran T$, and $\ker T' \supseteq \ker T$.  There is
  also a Hilbert space $\HH'' \supseteq \HH'$ and $T'' \in \mc{P}^2$
  on $\HH''$, such that $\HH'$ is invariant for ${T''}^*$, $\HH$ is
  semi-invariant for $T''$, and $T$ is the compression of $cT''$ for
  some $c > 0$.
\end{prop}

\begin{proof}
  Suppose that $T = AB$, where $(A,B)$ is optimal and by scaling if
  necessary, that $\|B\| \leq 1$.  On $\HH' = \HH \oplus \HH$, the
  operator
  \begin{equation*}
    \tilde{B} :=
    \begin{pmatrix}
      B & B^{1/2}(1-B)^{1/2} \\ (1-B)^{1/2}B^{1/2} & 1-B
    \end{pmatrix}
    =
    \begin{pmatrix}
      B^{1/2} \\ (1-B)^{1/2}
    \end{pmatrix}
    \begin{pmatrix}
      B^{1/2} & (1-B)^{1/2}
    \end{pmatrix}
  \end{equation*}
  is seen to be a projection since the column operator is an isometry.
  Extend $A$ to $\tilde{A}$ by padding with $0$s.  Then
  $T' := \tilde{A}\tilde{B} \in \mc{L}^+ \cdot \mc{P}$ and
  \begin{equation*}
    T' =
    \begin{pmatrix}
      AB & AB^{1/2}(1-B)^{1/2} \\
      0 & 0
    \end{pmatrix}.    
  \end{equation*}
  Clearly, $\HH$ is invariant for $T'$ and $T = P_\HH T' |_\HH$.
  Also, $\clran T \subseteq \clran T' \subseteq \clran A = \clran T$,
  so equality holds throughout.  It is obvious that if $f\in \HH$ is
  in $\ker T$, it is in $\ker T'$.

  The operator $T''$ is constructed by applying the same method to
  $c{T'}^*$, where $c$ is chosen so that $\|c{T'}^*\| \leq 1$.
\end{proof}

Let $T \in \LL$.  Using the L\"owner order, define a partial order on
the set of optimal pairs for $T$ by
\begin{equation*}
  (A_{\alpha}, B_{\alpha}) \prec (A_{\beta}, B_{\beta})
\end{equation*}
if $A_{\alpha} \leq A_{\beta}$ and $B_{\alpha} \leq B_{\beta}$.

\begin{Def} Let $T\in\LL$.  An optimal pair for $T$,
  $(A_{min},B_{min})$, is said to be \emph{minimal} if for an optimal
  pair $(A,B)$, $(A,B) \prec (A_{min}, B_{min})$ implies that
  $(A,B) = (A_{min}, B_{min})$.
\end{Def}

\begin{prop}
  \label{prop:min-min-opt-pairs}
  Let $T \in \LL$.  For every optimal pair $(A,B)$ for $T$, there
  exists a minimal optimal pair $(A_{min},B_{min}) \prec (A,B)$.
\end{prop}

\begin{proof}
  Suppose that with respect to the partial order $\prec$,
  $(A_{\lambda}, B_{\lambda})_{\lambda \in \Lambda}$ is a chain in the
  collection of optimal pairs for $T$.  Then the decreasing nets of
  positive operators $(A_\lambda)_{\lambda}$, $(B_\lambda)_{\lambda}$
  converge strongly to some $A, B \in \mc{L}^+$, respectively, and
  $T = AB$.  Since $\ker B_\lambda = \ker T$,
  $\ker T \subseteq \ker B$, and since $T = AB$, equality holds.
  Likewise, $\ker T^* = \ker A$.  Hence $(A,B)$ is optimal.  Thus
  every chain has a lower bound, and so minimal optimal pairs exist by
  Zorn's lemma.
\end{proof}

\begin{obs}
  For any minimal optimal pair $(A,B)$, $A = A_{\clran T^*}$ and $B =
  B_{\clran T}$.  So $A = F^*F$, $B = G^*G$, where
  \begin{equation*}
    F =
    \begin{pmatrix}
      F_1 & F_2 \\ 0 & 0
    \end{pmatrix}
    \quad\text{and} \quad
    G =
    \begin{pmatrix}
      G_1 & G_2 \\ 0 & 0
    \end{pmatrix}
  \end{equation*}
  on $\clran T^* \oplus \ker T$ and $\clran T \oplus \ker T^*$,
  respectively.

  Minimal optimal pairs need not be unique.  As a simple example, let
  $R > 1$ on $\HH$, and $T = R \oplus R^{-1}$ on $\HH \oplus \HH$.
  Then for $A = R \oplus 1$, $B = 1 \oplus R^{-1}$, both $(A,B)$ and
  $(B,A)$ are minimal optimal pairs for $T$.
\end{obs}

Lemma~\ref{lemaT1a} already hints that operators in $\LL$ share
certain properties with positive operators, many more of which will be
explored in the next section.  It is reasonable to wonder if an
operator in $\LL$ has a square root in $\LL$.  Partial results in this
direction are given next.  First, recall the following result of
Pedersen and Takesaki~\cite{MR306958} (slightly rephrased).

\begin{prop}[Pedersen-Takesaki]
  \label{prop:pedersen-takesaki}
  Let $H,K \in \mc{L}^+$, and write $\KK$ for $\clran H$.  A necessary
  and sufficient condition for the existence of $X \in \mc{L}^+$ such
  that $P_{\KK} K P_{\KK} = X H X$ is that $(H^{1/2} K H^{1/2})^{1/2}
  \leq a H$ for some $a \geq 0$.
\end{prop}

Though they do not show it, under the conditions of the proposition,
with respect to the decomposition $\HH = \mc{S}_1 \oplus \mc{S}_2$,
$\mc{S}_1 = \mc{S}_2 = \KK$, $X$ can be chosen as the $(2,2)$ entry of
the $\mc{S}_1$-compression of
\begin{equation*}
  \begin{pmatrix}
    aH & (H^{1/2} K H^{1/2})^{1/4} \\ (H^{1/2} K H^{1/2})^{1/4} & 1
  \end{pmatrix}
  \geq 0.
\end{equation*}

\begin{prop}
  \label{prop:qa-pos-and-sq-roots}
  Let $T \in \LL$ and suppose that $T$ is quasi-affine to a positive
  operator, $TX = XC$.  If $C^{1/2} \leq a X^*X$ for some $a \geq 0$,
  then $T$ has a square root in $\LL$.  Otherwise, if
  $\clran T = \clran T^*$ and $T$ has a factorization satisfying the
  conditions of Proposition~\ref{prop:pedersen-takesaki}, then $T$
  admits a square root in $\LL$.
\end{prop}

\begin{proof}
  If $C^{1/2} \leq a X^*X$, by Douglas' lemma, $C^{1/4} = X^*F$.
  Hence $C^{1/2} = X^*GX$, where $G \geq 0$, and so
  $TX = (XX^*)G(XX^*)GX$.  Since $\ran X$ is dense, $T = ((XX^*)G)^2$.

  Now suppose instead that $\clran T = \clran T^*$ and $T$ has a
  factorization satisfying the conditions of
  Proposition~\ref{prop:pedersen-takesaki}.  Choose $H = B$ and
  $K = A$ in Proposition~\ref{prop:pedersen-takesaki}.  Then $A =
  XBX$ for some $X \in \mc{L}^+$, and so $T = (XB)^2$.
\end{proof}

It was already noted in Theorem~\ref{thmLL} that if $T$ is similar to
a positive operator, it has a square root in $\LL$.  The next example
shows that this fails more generally.  The explanation requires a
lemma showing that an injective positive operator has a unique square
root in $\LL$ -- namely the positive square root.

\begin{lema}
  \label{lem:sq-root-of-pos-op-unique-in-L+2}
  If $A \in \mc{L}^+$ is injective, $R^2 = A^2$, and $RX = XA$ for
  some quasi-affinity $X$, then $R = A$.  Consequently, if $R\in \LL$
  satisfies $R^2 = A^2$, then $R = A$.
\end{lema}

\begin{proof}
  Suppose that $A \in \mc{L}^+$ is injective, $R^2 = A^2$, and
  $RX = XA$, where $X$ is a quasi-affinity.  Then
  $A^2X = R^2X = XA^2$, and so $AX = XA$ since $A$ is in the
  commutative $C^*$-algebra generated by $A^2$.  Thus $RX = AX$, and
  therefore $R = A$.

  Write $R = XY$, $X,Y \in \mc{L}^+$ and injective.  Then
  $(XY)^2 = R^2 = A^2 = R^{*\,2} = (YX)^2$.  Hence
  $A^2 X^{1/2} = X^{1/2}(X^{1/2} Y X^{1/2})^2$, and by
  \cite[Lemma~4.1]{MR250106}, there is a unitary $U$ such that
  $A^2 = U(X^{1/2} Y X^{1/2})^2 U^*$, and so
  $A = U(X^{1/2} Y X^{1/2}) U^*$.  Thus
  \begin{equation*}
    R(X^{1/2} U^*) = (XY)(X^{1/2} U^*) = (X^{1/2} U^*) A.
  \end{equation*}
  Since $X^{1/2} U^*$ is a quasi-affinity, $R = A$.
\end{proof}

\begin{Example}
  \label{ex:T-in-L+2-but-no-sq-rt-in-L+2}
  Following ideas from \cite{MR2775769}, let $T \in \mc{P}^2$ be the
  product of $A$ and $B$, non-trivial projections, given as follows.
  With respect to the decomposition
  $\HH = \ran B \oplus (\ran B)^\bot$,
  \begin{equation*}
    T = AB =
    \begin{pmatrix}
      S & S^{1/2}(1-S)^{1/2} \\ (1-S)^{1/2}S^{1/2} & 1-S
    \end{pmatrix}
    \begin{pmatrix}
      1 & 0 \\ 0 & 0
    \end{pmatrix}
    =
    \begin{pmatrix}
      S & 0 \\ (1-S)^{1/2}S^{1/2} & 0
    \end{pmatrix},
  \end{equation*}
  where $S \geq 0$ is injective but not invertible on $\ran B$ with
  $\|S\| < 1$, so that $(1-S)^{1/2}$ is invertible.  Then
  $\ker T = \ker B$, and $\clran T = \ran G$,
  $G = \begin{pmatrix} S^{1/2} & 0 \\ (1-S)^{1/2} & 0 \end{pmatrix}$,
  and so equals $\ran A$ since $A = GG^*$.  This has zero
  intersection with $\ker T$ and
  $\clran T \dotplus \ker T = \ran S \oplus \ker B$, which is dense in
  $\HH$.  As an aside, in Corollary~\ref{propqs2} it will be shown
  that this condition implies that $T$ is quasi-similar to a positive
  operator.

  Suppose that $T = (XY)^2$ for some $X, Y \in \mc{L}^+$; that is, $T$
  has a square root in $\LL$.  Without loss of generality, $(X,Y)$ can
  be chosen to be an optimal pair for $R:= XY$.  Clearly,
  $\ker R \subseteq \ker T$, and since $\ran R \cap \ker R = \{0\}$,
  the reverse containment also holds.  Hence $\clran Y = \ran B$.  A
  similar calculation gives $\clran T = \ran A$.

  Now
  \begin{equation*}
    R^2 = {\left(
        \begin{pmatrix}
          X_{11} & X_{12} \\ X_{12}^* & X_{22}
        \end{pmatrix}
        \begin{pmatrix}
          Y_{11} & 0 \\ 0 & 0
        \end{pmatrix}
      \right)}^2
    =
    \begin{pmatrix}
      (X_{11}Y_{11})^2 & 0 \\ (X_{12}^* Y_{11})(X_{11}Y_{11}) & 0
    \end{pmatrix}.
  \end{equation*}
  Thus $(X_{11}Y_{11})^2 = S$, and so by
  Lemma~\ref{lem:sq-root-of-pos-op-unique-in-L+2},
  $X_{11}Y_{11} = S^{1/2}$.  Since
  $(1-S)^{1/2} = X_{12}^* Y_{11} = Y_{11} X_{12}$ is invertible and
  $Y_{11}$ is injective, $Y_{11}$ is invertible by the open mapping
  theorem, and hence the same is true for $X_{12}$.  The operator $S$
  is not invertible, so $X_{11}$ is not invertible.  However, since
  $X \geq 0$, $X_{12} = X_{11}^{1/2}G$, and so $\ran X_{11}^{1/2}$ is
  closed.  This implies that $X_{11}^{1/2}$ is invertible, giving a
  contradiction.
\end{Example}

\section{Spectral properties of $\LL$}
\label{sec:l+2-spectral-properties}

Recall by Theorem~\ref{thmLL}, any operator which is similar to a
positive operator (and so in $\LL$) is necessarily scalar (that is, it
is spectral and has no quasi-nilpotent part).  It will be shown
further that finite rank operators in $\LL$ are completely
characterized by the property that the spectrum is positive and the
operator is diagonalizable (Corollary~\ref{corCR}).  It has already
been noted that operators in $\LL$ need not be quasi-affine to a
positive operator, much less similar to one, and as a result they are
in general not spectral.  Despite this, the spectral properties of
operators in $\LL$ are found to reflect what is observed in these
special cases.

The spectrum $\sigma(T)$ of an operator $T$ can be divided into two,
potentially overlapping parts; the \emph{compression spectrum}
$\sigma_c(T)$, points $\lambda$ of which have the property that
$T - \lambda 1$ is not surjective, and the approximate point spectrum
$\sigma_a(T)$, in which $T-\lambda 1$ is not bounded below.  The
subset of $\sigma_a(T)$ of points $\lambda$ for which $T-\lambda 1$ is
not injective constitute the point spectrum $\sigma_p(T)$.  Standard
results in operator theory are that $\lambda \in \sigma_p(T)$ is
equivalent to $\overline{\lambda} \in \sigma_c(T^*)$, and that the
topological boundary of the spectrum is contained in $\sigma_a(T)$.
In the case of $T \in \LL$, where the spectrum lacks interior, this
means that $\sigma(T) = \sigma_a(T)$.

The parts of the spectrum already mentioned are for the most part
enough when studying normal operators on Hilbert spaces.  Outside of
this class, it helps to refine these by looking at \emph{local
  spectral properties}.  This is ordinarily developed for (potentially
unbounded) Banach space operators, though here bounded Hilbert space
operators are solely considered.

Let $T\in L(\HH)$.  If a point $\mu$ is in $\rho(T)$, the resolvent of
$T$, $T - \mu 1$ is invertible.  Equivalently, for all $x \in \HH$ and
$\lambda \in U$, an open neighborhood of $\mu$,
$f(\lambda) = (T- \lambda 1)^{-1} x$ is an analytic function from $U$
into $\HH$ and satisfies $(T-\lambda 1)f(\lambda) = x$.  Even if
$\mu \notin \rho(T)$, it may happen that for some $x\in \HH$ and
neighborhood $U$ of $\mu$, there is an analytic $f:U \to \HH$ such
that $(T -\lambda 1)f(\lambda) = x$.  In this case,
$\mu \in \rho_T(x)$, the \emph{local resolvent of $T$ at $x$}.  For
fixed $x$, the complement in $\mathbb C$ of $\rho_T(x)$ is called the
local spectrum of $T$ at $x$, and is denoted by $\sigma_T(x)$.

An operator $T$ is said to have the \emph{single valued extension
  property} (abbreviated \emph{SVEP}) if whenever
$U\subseteq \mathbb C$ is open and $f:U \to \HH$ is an analytic
function satisfying $(T - \lambda 1) f(\lambda) = 0$ for all
$\lambda \in U$, then $f = 0$.  The point of SVEP is that if $T$ has
this property, any solution $f$ to $(T-\lambda 1)f(\lambda) = x$ in a
neighborhood of a point $\mu$ is unique.  Operators like those in
$\LL$ with thin spectrum have SVEP.

For $F \subseteq \mathbb C$ closed, an \emph{(analytic) local spectral
  subspace for $T \in L(\HH)$} is defined as
\begin{equation*}
  \HH_T(F) := \{x\in \HH : \sigma_T(x) \subseteq F\}.
\end{equation*}
This is a (not necessarily closed) linear manifold.  Properties
include that $\HH_T(F) = \HH_T(\sigma (T) \cap F)$, and if $T$ has
SVEP, $\HH_T(\emptyset) = \{0\}$.  Hence for operators in $\LL$, it
will suffice to consider $\HH_T(F)$ for closed subsets of
$\sigma (T)$.  It is also the case that for $\lambda \notin F$,
$(T-\lambda 1)\HH_T(F) = \HH_T(F)$, $\HH_T(F)$ is invariant for all
operators commuting with $T$ (in other words, it is
\emph{hyperinvariant}).  Also, for all $n\in \mathbb N$ and
$\lambda \in \mathbb C$,
$\ker (T - \lambda 1)^n \subseteq \HH_T(\{\lambda\})$, and more
generally, if for $x\in \HH$ and $\lambda \in F$,
$(T - \lambda 1)x \in \HH_T(F)$, then $x\in \HH_T(F)$.  See
\cite[Proposition~1.2.16]{MR1747914}.  By
\cite[Proposition~1.3]{MR1004421}, when $T$ has SVEP,
$\HH_T(\{\lambda\}) = \{x : \lim_n \|(T - \lambda 1)^n x\|^{1/n} =
0\}$.

The following is a special case of a result due to Putnam, and Pt\'ak
and Vrbov\'a (see \cite[Theorem~1.5.7]{MR1747914}).  The proof in this
case is elementary and is included for completeness.  The more general
result is discussed below.

\begin{lema}
  \label{lem:X-of-T-of-lambda-for-normal-op}
  Let $T \in L(\HH)$ be normal and $\lambda \in \mathbb C$.  Then
  $\HH_T(\{\lambda\}) = \ker(T-\lambda 1)$.
\end{lema}

\begin{proof}
  Recall that for a normal operator $T$, the norm equals the spectral
  radius:
  \begin{equation*}
    \|T\| = \lim_{n\to\infty} \|T^n\|^{1/n}.
  \end{equation*}
  Also, $T$ is spectral so has SVEP.  Let
  $\HH_T(\{\lambda\}) = \{x: \lim_{n\to\infty} \|(T - \lambda 1)^n
  x\|^{1/n} \to 0\}$, and $\mc{E} = \overline{\HH_T(\{\lambda\})}$.
  Then $(T-\lambda 1)\mc{E} \subseteq \mc{E}$, so
  $T\mc{E} \subseteq \mc{E}$.  Since for all $y$,
  $\|(T-\lambda 1)y)\| = \|(T^*-\lambda 1)y)\|$,
  $T^*\mc{E} \subseteq \mc{E}$.  Thus $\mc{E}$ reduces $T$, and
  $T_0:= P_{\mc{E}} T |_{\mc{E}}$ is normal.

  Let $x \in \HH_T(\{\lambda\})$, $\|x\| = 1$.  Then for all
  $\epsilon > 0$, for sufficiently large $n$,
  $\|(T_0-\lambda 1_{\mc{E}})^n x \| < \epsilon^n$.  So if
  $y \in \mc{E}$ with $\|y\| = 1$,
  \begin{equation*}
    \epsilon^n > \PI{(T_0-\lambda 1_{\mc{E}})^n x}{y} =
    \PI{x}{(T_0-\lambda 1_{\mc{E}})^{*\,n} y}.
  \end{equation*}
  Since $\HH_T(\{\lambda\})$ is dense in $\mc{E}$ and $\epsilon$ is
  arbitrary, $\|(T_0-\lambda 1_{\mc{E}})^{*\,n} y\|^{1/n} \to 0$ for
  all $y \in \mc{E}$.  Thus
  $\sigma(T_0-\lambda 1_{\mc{E}}) = \sigma((T_0-\lambda 1_{\mc{E}})^*)
  = \{0\}$.  Hence by normality, $T_0-\lambda 1_{\mc{E}} = 0$, and so
  $\HH_T(\{\lambda\}) = \ker (T - \lambda 1)$.
\end{proof}

\begin{prop}
  \label{prop:X-sub-T-of-lambda-closed}
  For $T\in \LL$ and $\lambda \in \mathbb C$,
  $\HH_T(\{\lambda\}) = \ker(T - \lambda 1)$.
\end{prop}

\begin{proof}
  By definition,
  $\HH_T(\{\lambda\}) = \{x: \sigma_T(x) = \{\lambda\}\} = \{x:
  \sigma_{T-\lambda 1}(x) = \{0\}\} \supseteq \ker(T - \lambda 1)$.  If
  $\lambda\in \rho(T)$, then $T-\lambda 1$ is invertible, and so for
  all $x\neq 0$, $\rho_T(x) \supseteq \rho(T)$, or equivalently,
  $\sigma_T(x) \subseteq \sigma(T)$.  In particular then, if
  $\lambda \in \rho(T)$,
  $\HH_T(\{\lambda\}) = \{0\} = \ker(T - \lambda 1)$.

  So suppose that $\lambda \geq 0$ is in $\sigma (T)$.  Write $T = AB$
  for some optimal pair $(A,B)$, and set $C = B^{1/2}AB^{1/2}$.  Then
  $B^{1/2}(T-\lambda 1) = (C-\lambda 1)B^{1/2}$, and by induction,
  $B^{1/2}(T-\lambda 1)^n = (C-\lambda 1)^n B^{1/2}$ for
  $n \in \mathbb N$.  Let
  $x \in \HH_T(\{\lambda\}) = \{y : \lim_n \|(T - \lambda 1)^n
  y\|^{1/n} = 0\}$.  Then
  \begin{equation*}
    \|B^{1/2}(T-\lambda 1)^n x\|^{1/n} \leq \|B^{1/2}\|^{1/n}
    \|(T-\lambda 1)^n x\|^{1/n} \to 1 \cdot 0 = 0,
  \end{equation*}
  and so
  \begin{equation*}
    \|(C-\lambda 1)^n B^{1/2}x\|^{1/n} \to 0.
  \end{equation*}
  Thus $B^{1/2} x \in \HH_C(\{\lambda\})$.  By the previous lemma
  $\HH_C(\{\lambda\}) = \ker(C-\lambda 1)$, hence
  \begin{equation*}
    B^{1/2}(T-\lambda 1)x = (C-\lambda 1) B^{1/2} x = 0.
  \end{equation*}
  
  If $\lambda = 0$, then either $x\in \ker T$ or
  $Tx \in \ker B = \ker T$.  But by Proposition~\ref{lemaT2},
  $\ran T \cap \ker T = \{0\}$, and so this also implies that
  $x\in \ker T$.  If $\lambda > 0$, then similar reasoning gives
  either $(T-\lambda 1) x = 0$ or
  $(T-\lambda 1) x \in \ker B = \ker T$.  Suppose the latter.  Since
  $(T-\lambda 1) Tx = T(T-\lambda 1)x = 0$, it follows that
  $(T-\lambda 1)^2 x = -\lambda (T-\lambda 1)x$, and in general, by
  induction for all $n$,
  \begin{equation*}
    (T-\lambda 1)^n x = (-1)^{n-1} \lambda^{n-1} (T-\lambda 1)x.
  \end{equation*}
  Hence
  \begin{equation*}
    \lambda^{(n-1)/n}\|(T-\lambda 1)x\|^{1/n} = \|\lambda^{n-1}
    (T-\lambda 1)x\|^{1/n} = \|(T - \lambda 1)^n x\|^{1/n}.
  \end{equation*}
  Since as $n \to \infty$, the right hand term goes to $0$,
  $\lambda^{(n-1)/n} \to \lambda > 0$, and
  $\|(T-\lambda 1)x\|^{1/n} \to 1$ if $\|(T-\lambda 1)x\| > 0$, the
  conclusion is that $x\in \ker (T-\lambda 1)$.
\end{proof}

A simplified version of the above argument can be used to show the
following.

\begin{prop}
  \label{prop:T-qa-normal-and-H_T-lambda}
  If $T\in L(\HH)$ is quasi-affine to a normal operator, then
  \begin{equation*}
    \HH_T(\{\lambda\}) = \ker(T-\lambda 1).
  \end{equation*}
\end{prop}

There are further ways in which operators in $\LL$ resemble positive
operators.  To explain this requires the introduction of some
additional ideas from local spectral theory, details for which can be
found in~\cite{MR1747914} and~\cite{MR0394282}.

Recall that a scalar operator is one which is similar to a normal
operator, and so has a Borel functional calculus.  By
Theorem~\ref{thmLL}, if $T \in \LL$ and is similar to a positive
operator, then it is scalar, and so also has a Borel functional
calculus.  An operator $T$ is termed a \emph{generalized scalar
  operator} if it has a $C^\infty$ functional calculus; that is, there
is a continuous homomorphism $\Phi: C^\infty(\mathbb C) \to L(\HH)$
with $\Phi(1) = 1$ and $\Phi(z) = T$.  An operator which is the
restriction of a generalized scalar operator to an invariant subspace
is said to be \emph{subscalar}.  Obviously, the classes of generalized
scalar and subscalar operators include that of scalar operators.

\begin{thm}
  \label{prop:T-in-L+2-gen-scalar}
  Let $T \in \LL$.  Then $T$ is a generalized scalar operator and $T$
  has a $C^2([0,\|T\|])$ functional calculus.
\end{thm}

\begin{proof}
  To begin with, claim that if either $A$ or $B$ has closed range,
  then $T$ is generalized scalar.  So suppose $T = AB$,
  $A,B \in \mc{L}^+$, where $\ran A$ is closed (the case where
  $\ran B$ is closed can be handled identically by working with
  $T^*$).  Decompose $\HH = \ran A \oplus (\ran A)^\bot$, and write
  \begin{equation*}
    T =
    \begin{pmatrix}
      T_1 & T_2 \\ 0 & 0
    \end{pmatrix}
  \end{equation*}
  with respect to this decomposition.  By the assumption that $\ran A$
  is closed, $T_1$ is similar to a positive operator, and so is scalar
  by Theorem~\ref{thmLL}.
  Hence there is a constant $\kappa \geq 1$ such that for
  $\lambda \in \mathbb C \backslash \mathbb R$,
  \begin{equation*}
    \|(T_1 - \lambda 1)^{-1}\| \leq \kappa (1 +
    |\mathrm{Im}\,\lambda|^{-1}).
  \end{equation*}
  Since
  \begin{equation*}
    (T-\lambda 1)^{-1} =
        - \lambda \begin{pmatrix}
          (T_1 - \lambda 1)^{-1} & \tfrac{1}{\lambda}(T_1 - \lambda
          1)^{-1} \\
          0 & -\tfrac{1}{\lambda}
        \end{pmatrix},
  \end{equation*}
  it follows that for sufficiently large $\kappa'$,
  \begin{equation*}
    \begin{split}
      \|(T-\lambda 1)^{-1}\|
      & \leq
      \kappa (1 + |\mathrm{Im}\,\lambda|^{-1})
      \left(1 + |\mathrm{Im}\,\lambda|^{-1}
        \left\| \begin{pmatrix}
            0 &  T_2 \\
            0 &  0
          \end{pmatrix} \right\|\right) \\
      & \leq \kappa' |1 + |\mathrm{Im}\,\lambda|^{-2}|.
    \end{split}
  \end{equation*}
  From \cite[Theorem~1.5.19]{MR1747914}, $T$ is a generalized scalar
  operator.

  For the general case, let $T \in \LL$ and let
  $T' \in \mc{L}^+ \cdot \mc{P}$ on $\HH'$ be the dilation of $T$ from
  Proposition~\ref{prop:dilations-for-L+2}.  So $\HH$ is an invariant
  subspace for $T'$ and $T$ is the restriction of $T'$ to $\HH$.
  Hence $T$ is subscalar.

  Subscalar operators need not be generalized scalar.  However, in
  this case any $\lambda \in \mathbb C \backslash \mathbb R$ is in the
  resolvents of both $T$ and $T'$.  So writing $T' =
  \begin{pmatrix} T & T_2 \\ 0 & 0 \end{pmatrix}$ with respect to the
  decomposition $\HH' = \HH \oplus \HH^\bot$, for
  $\lambda \in \mathbb C \backslash \mathbb R$,
  \begin{equation*}
    (T' - \lambda 1_{\HH'})^{-1} =
    \begin{pmatrix}
      (T - \lambda 1_{\HH})^{-1} & \tfrac{1}{\lambda }(T - \lambda
      1_{\HH})^{-1}T_2 \\
      0 & -\tfrac{1}{\lambda}
    \end{pmatrix}.
  \end{equation*}
  Consequently, for all $\lambda \in \mathbb C \backslash \mathbb R$,
  there is a $\kappa' > 0$ such that
  \begin{equation*}
    \|(T - \lambda 1_{\HH})^{-1}\| \leq \|(T' - \lambda
    1_{\HH'})^{-1}\| \leq \kappa' (1 + |\mathrm{Im}\,\lambda|^{-2})
  \end{equation*}
  by the first part of the proof.  Thus
  \cite[Theorem~1.5.19]{MR1747914} gives that $T$ is a generalized
  scalar operator.  The fact that $T$ has a $C^2([0,\|T\|])$
  functional calculus follows from the proof of that theorem.
\end{proof}

Let $F \subset \mathbb C$.  For an operator $T$, the \emph{algebraic
  spectral subspace $\mc{E}_T(F)$} is the largest linear manifold such
that $(T-\lambda 1) \mc{E}_T(F) = \mc{E}_T(F)$ for all
$\lambda \notin F$.  Moreover, for every positive integer $p$,
\begin{equation*}
  \HH_T(F) \subseteq \mc{E}_T(F) \subseteq \bigcap_{\lambda \notin F}
  \ran (T- \lambda 1)^p.
\end{equation*}
When $T$ is normal and $E_T(F)$ is the spectral projection for the set
$F$, it turns out that
$\HH_T(F) = \mc{E}_T(F) = \bigcap_{\lambda \notin F} \ran (T- \lambda
1) = \ran E_T(F)$~\cite[Theorem~1.5.7]{MR1747914}.  The next result
states that the operators in $\LL$ behave in this respect like normal
operators.

\begin{prop}
  \label{prop:alg-spec-subsp-for-L+2}
  Let $T \in \LL$.  Then for closed $F \subset \mathbb C$,
  \begin{equation*}
    \HH_T(F) = \mc{E}_T(F) = \bigcap_{\lambda \notin F} \ran (T- \lambda
    1).
  \end{equation*}
\end{prop}

\begin{proof}
  By Theorem~\ref{prop:T-in-L+2-gen-scalar}, $T\in \LL$ is a
  generalized operator, so by \cite[Theorem~1.5.4]{MR1747914}, there
  exists an integer $p$ such that for any closed set $F$,
  $\HH_T(F) = \mc{E}_T(F) = \bigcap_{\lambda \notin F} \ran (T-
  \lambda 1)^p$.  Fix $\lambda \notin F$.  Since $T^* \in \LL$, by
  Proposition~\ref{prop:X-sub-T-of-lambda-closed},
  $\ker(T^* - \overline{\lambda} 1)^p = \ker(T - \lambda 1)^{*\,p} =
  \ker(T^*-\overline{\lambda} 1)$ for all $p\in \mathbb N$, and so
  $\ran (T - \lambda 1)^p = \ran (T - \lambda 1)$ for all $p$.
\end{proof}

\section{$\LL$ and similarity; the set $\LLcr$}
\label{sec:l+2-and-sim--crl+2}

In Proposition~\ref{lemaT2}, it was proved that if $T \in \LL$ then
$\ran T \cap \ker T = \{0\}$.  It is always the case then that
\begin{equation*}
  \HH = \ol{\ran T \dotplus \ker T} \oplus (\ker T^* \cap \clran T^*).
\end{equation*}
This section considers the case where $\ran T \dotplus \ker T$ is
dense in $\HH$.  In Section~\ref{sec:l+2--general-case}, the general
case will be taken up.

Recall from Proposition~\ref{thmLL}, $T\in \LL$ and $T$ admits a
factorization $T = AB$ where $A,B \in \mc{L}^+$ and either $A$ or $B$
is invertible is equivalent to $T$ being similar to a positive
operator.

\begin{prop}
  \label{optimalinv}
  Let $T \in \LL$ and $A \in \mc{L}^+$ such that
  $\clran A = \clran T$.  Then the following are equivalent:
  \begin{enumerate}
  \item There exists $B \in GL(\HH)^+$ such that $T = AB$;
  \item There exists $B \in \mc{L}^+$ such that $(A,B)$ is
    optimal for $T$ and $\ran B \dotplus \ker A = \HH$;
  \item $\mc{B}^A_T \neq \emptyset$ and $\ran T = \ran A$.
  \end{enumerate}
  As a result, for this choice of $A$, there is a unique optimal pair
  $(A,B_0)$ and $B_0$ has closed range.
\end{prop}

\begin{proof}
  $(\mathit{i}\kern0.5pt) \Rightarrow (\mathit{ii}\kern0.5pt)$:
  Suppose that $T = AB$ with $B \in GL(\HH)^+$ then $\ran T = \ran A =
  \ran (AB')$ for any optimal pair $(A,B')$.  Then $\HH = A^{-1}
  \ran(AB') = \ran B' \dotplus \ker A$, where the sum is direct by
  Proposition~\ref{lemaT2}.

  $(\mathit{ii}\kern0.5pt) \Rightarrow (\mathit{iii}\kern0.5pt)$:
  Suppose that there exists $B \in \mc{L}^+$ such that $(A,B)$ is
  optimal for $T$ and $\HH = \ran B \dotplus \ker A$.  Applying $A$ to
  both sides gives $\ran A = \ran(AB) = \ran T$.

  $(\mathit{iii}\kern0.5pt) \Rightarrow (\mathit{i}\kern0.5pt)$: Let
  $(A,B')$ be an optimal pair for $T$.  Such a pair exists by
  Proposition~\ref{lemaT2}.  Since $\ran T = \ran A$, by the same
  calculation as above, $\HH = \ran B' \dotplus \ker A$.  Then by
  \cite[Theorem~2.3]{MR293441}, which states that if an operator range
  is complemented, then it is closed, $\ran B'$ is closed.

  Now define the positive operator $B = B' + P_{\ker A}$.  By
  \cite[Theorem~2.2]{MR293441},
  $\ran B^{1/2} = \ran B' + \ran P_{\ker A} = \HH$, and so $B$ is
  invertible.

  The last statement follows from Corollary~\ref{coruniqueBTA}.
\end{proof}

Theorem~\ref{thmLL} indicates a number of ways of finding operators
which are similar to positive operators.  In addition, it combines
with the last result to give yet another.

\begin{cor}
  \label{coroptimalclosed}
  Let $T \in L(\HH)$.  Then $T$ is similar to a positive operator if
  and only if $T \in \LL$, $\clran T \dotplus \ker T = \HH$ and there
  exists and optimal pair $(A,B)$ such that either $A$ or $B$ has
  closed range.
\end{cor}

It is not true in general that if $T$ is similar to a positive
operator, then every optimal pair for $T$ is such that one of its
factors has closed range.

\begin{Example}
  \label{example-3}
  Let $A \in \mc{L}^+$ be such that $\ran A$ is not closed.  Then
  clearly $A$ is similar to a positive operator and
  $(A^{1/2}, A^{1/2})$ is an optimal pair for $A$.  But, since
  $\ran A \subsetneq \ran A^{1/2}$, then none of the factors of this
  optimal pair has closed range.  However, since $A = AP_{\clran A}$,
  the optimal pair $(A,P_{\clran A})$ is as in
  Corollary~\ref{coroptimalclosed}.
\end{Example}

The situation when the range of $T \in \LL$ is closed happens to be
special as well.  Write
\begin{equation*}
  \LLcr := \{ T \in \LL: T \mbox{ has closed range} \}.
\end{equation*}

\begin{prop}
  \label{propCR1}
  Let $T \in \LL$.  Then the following are equivalent:
  \begin{enumerate}
  \item $T \in CR(\HH)$;
  \item $\ran T \dotplus \ker T = \HH$;
  \item For any optimal pair $(A,B)$, $A,B \in CR(\HH)$ and $\ran A
    \dotplus \ker B$ is closed.
  \end{enumerate}
  In this case, $T$ is similar to a positive operator.
\end{prop}

\begin{proof}
  Let $T \in \LL$ and suppose that $T \in CR(\HH)$.  Then
  $T^* \in CR(\HH)$.  Let $(A,B)$ be an optimal pair.  Then
  $\ran A \supseteq \ran T = \clran T = \clran A$, and similarly,
  $\ran B = \clran B$.  Thus
  $A,B \in CR(\HH)$ and $\HH = B^{-1} \ran T^* = \ran A \dotplus \ker B
  = \ran T \dotplus \ker T$.  Conversely, if $\ran T \dotplus \ker T =
  \HH$, by \cite[Theorem~2.3]{MR293441}, $\ran T$ is closed.

  Finally, suppose that for an optimal pair $(A,B)$, $A,B \in CR(\HH)$
  and $\ran A \dotplus \ker B$ is closed.  By
  \cite[Corollary~2.5]{MR651705}, $\ran T = \ran(AB)$ is closed.  On
  the other hand, if $\ran T \dotplus \ker T = \HH$, then arguing as
  above, $\ran A \dotplus \ker B = \HH$, and so is closed.  Hence all
  of the items are equivalent.

  The statement that $T$ is similar to a positive operator follows
  from Corollary~\ref{coroptimalclosed}.
\end{proof}

Proposition~\ref{propCR1} and Theorem~\ref{thmLL} together imply the
following.

\begin{cor}
  \label{corCR2}
  Let $T \in L(\HH)$.  The following are equivalent:
  \begin{enumerate}
  \item $T \in \LLcr$;
  \item $T = ST^*S^{-1}$ with $S \in GL(\HH)^+$ and
    $\sigma(T) \subseteq \{0\} \cup [c,\infty)$ for $c > 0$;
  \item There exists $G \in GL(\HH)$ such that
    $GTG^{-1} \in CR(\HH)^+$;
  \item $T$ is a scalar operator and
    $\sigma(T) \subseteq \{0\} \cup [c,\infty)$ for $c > 0$.
  \end{enumerate}
\end{cor}

If $T \in \LLcr$, then by Proposition~\ref{propCR1}, $T$ is similar to
a positive operator $C$.  From this it is not difficult to check that
$\sigma(T) = \sigma(C)$, $C$ also has closed range, and consequently
the spectrum of both operators have the form indicated in the
corollary.

\begin{cor}
  \begin{equation*} \LLcr = \bigcup_{W \in CR(\HH)^+} \mathbb{O}_W.
  \end{equation*}
\end{cor}

\begin{cor}
  \label{corCR}
  Suppose that $T \in L(\HH)$ is finite rank.  Then $T \in \LL$ if and
  only if $T$ is diagonalizable and $\sigma(T) \geq 0$.
\end{cor}

\begin{obs}
  \label{obs:formula-for-AB-in-fd-case}
  If $T$ on $\HH$ with $dim(\HH) < \infty$ is diagonalizable with
  positive spectrum, it is in principle straightforward to write $T$
  as a product of two positive operators.  Let $C$ be the diagonal
  matrix of eigenvalues of $T$, $V$ the matrix with columns consisting
  of the eigenvectors of $T$, arranged in the same order as the
  diagonal entries of $C$.  The matrix $V$ is invertible, and
  $TV = VC$.  Therefore, $T = (VV^*)(V^{*\,-1} C V^{-1})$.
\end{obs}

\begin{Example}
  \label{example-4}
  The situation for the product of three or more positive operators is
  more complicated.  In particular, such products need not be
  diagonalizable.  As a simple example,
  \begin{equation*}
    \begin{pmatrix}
      1 & 0 \\ 0 & 0
    \end{pmatrix}
    \begin{pmatrix}
      1 & 1 \\ 1 & 1
    \end{pmatrix}
    \begin{pmatrix}
      0 & 0 \\ 0 & 1
    \end{pmatrix}
    =
    \begin{pmatrix}
      0 & 1 \\ 0 & 0
    \end{pmatrix}.
  \end{equation*}
  Hence the class $\mc{L}^{+3}$ of products of three positive
  operators strictly contains $\LL$.  

  Maganja showed in \cite{MR3008882} that every bounded operator on a
  Hilbert space is the sum of at most three operators which are
  similar to positive operators (and so by Theorem~\ref{thmLL}, three
  operators in $\LL$).  On finite dimensional spaces, Wu proved that
  if $\det T \geq 0$ (which includes those $T$ with non-negative
  spectrum), $T$ is the product of at most $5$ positive
  matrices~\cite{MR974046}, and in~\cite{MR3652833}, an algorithm is
  given for determining the number of matrices between $1$ and $5$
  needed.  In the setting of separable Hilbert spaces, Wu also showed
  that any operator which is the norm limit of a sequence of
  invertible operators is the product of at most $18$ positive
  operators~\cite{MR990101}.  For invertible operators, this was
  improved by Phillips to at most~$7$~\cite{MR1335103}.
\end{Example}

It is also possible to give an explicit formula for the Moore-Penrose
inverse $T^\dagger$ of an operator $T \in \LLcr$.  In this case, if
$Q := P_{\ran T^* /\!/ \ker T^*}$ is the oblique projection onto
$\ran T^*$ along $\ker T^*$, $Q$ is bounded.  Recall that an operator
$T'$ is called a $(1,2)$-inverse of $T$ if $TT'T = T$ and
$T'TT' = T'$.  Generally, there will be infinitely many
$(1,2)$-inverses for an operator $T$.  The Moore-Penrose inverse is
the $(1,2)$-inverse for which $TT^\dagger$ is the orthogonal
projection onto $\ran T$ and $T^\dagger T$ is the orthogonal
projection onto $\ran T^*$.

\begin{prop}
  \label{prop:MP-inverse}
  Let $T \in \LLcr$ with optimal pair $(A,B)$.  Then
  \begin{equation*}
    T^\dagger = B^\dagger Q A^\dagger.
  \end{equation*}
  Furthermore, $T' := Q^* T^\dagger Q^*$ is a $(1,2)$-inverse of $T$
  in $\LLcr$.
\end{prop}

\begin{proof}
  Let $T \in \LLcr$.  The fact that $A$ and $B$ have closed range
  follows from Proposition~\ref{propCR1}.  Hence $A^\dagger$ and
  $B^\dagger$ are bounded positive operators.  Also, $\ran T^*$ is
  closed.  For $Q = P_{\ran T^* /\!/ \ker T^*}$,
  $P_{\ran T^*} Q P_{\ran T} = Q P_{\ran T} = Q$.  For
  $W = B^\dagger Q A^\dagger$,
  \begin{equation*}
    TWT = AB(B^\dagger Q A^\dagger)AB = AP_{\ran T^*} Q P_{\ran T}B =
    AQB = T.
  \end{equation*}
  Therefore $TW$ is a projection.  Furthermore,
  \begin{equation*}
    TW = ABB^\dagger Q A^\dagger = AP_{\ran T^*} Q A^\dagger = A Q
    A^\dagger.
  \end{equation*}
  Also, $\ran (TW) = \ran T$ and $\ker (TW) = \ker T^*$ since
  \begin{equation*}
    \begin{split}
      \ran T &= \ran (TWT) \subseteq \ran (TW) \subseteq \ran T,
      \quad \text{and} \\
      \ker W & \subseteq \ker (TW) \subseteq \ker (WTW) = \ker W =
      \ker T^*.
    \end{split}
  \end{equation*}
  The last equality holds since if $x\in \ker W$,
  $Q A^\dagger x \in \ker T \cap \ran Q = \ker T \cap \ran T^* =
  \{0\}$, so
  $A^\dagger x \in \ker Q \cap \ran A^\dagger = \ker T^* \cap \ran T =
  \{0\}$.  Thus $x\in \ker A^\dagger = \ker T^*$.  Hence $\ran (TW)$
  and $\ker (TW)$ are orthogonal, and so $TW \in \mc{P}$.

  Similar calculations show that $WTW = W$, hence $WT$ is a
  projection, and by identical reasoning, it is an orthogonal
  projection.  Thus, $T^\dagger = B^\dagger Q A^\dagger$, as claimed.

  Since $Q^*T = TQ^* = T$, it is easy to see that for
  $T' = Q^* T^\dagger Q^*$, $TT'T = T$ and $T'TT' = T'$.  Also,
  \begin{equation*}
    \ran T' = Q^* T^\dagger \ran T = Q^* T^\dagger \HH = Q^* \ran T^*
    = Q^* \HH = \ran T.
  \end{equation*}
  Finally,
  \begin{equation*}
    T' = (Q^* B^\dagger Q)(Q A^\dagger Q^*) \in \LLcr.\qedhere
  \end{equation*}
\end{proof}

\begin{obs}
  \label{obs:MP-inv-for-P2}
  If $T\in \mc{P}^2$ with closed range, the formula $T^\dagger =
  P_{\ran T^* /\!/ \ker T^*}$ from \cite{MR2775769} is recovered.
\end{obs}

\section{$\LL$, quasi-affinity and quasi-similarity}
\label{sec:l+2-and-quasi-sim}

In Proposition~\ref{PropIa} it was seen that the statement that $T$
being quasi-affine to a positive operator is equivalent to, among
other things, being able to write $T^* = BA$ where $B$ and $A$ are
positive, but where $B$ may be unbounded.  The situation for
quasi-similarity is no better (Corollary~\ref{corpIa}).  Conditions
equivalent to $T = AB$ where $A$ and $B$ are bounded and positive
require something extra, and this will then imply $\sigma(T) \geq 0$
by Lemma~\ref{lemaT1a}.

\begin{thm}
  \label{PropI}
  For $T \in L(\HH)$, the following are equivalent:
  \begin{enumerate}
  \item $T \in \LL$ and is quasi-affine to a positive operator;
  \item $T \in \LL$ and $\ol{\ran T \dotplus \ker T} = \HH$;
  \item There exists a quasi-affinity $X \in \mc{L}^+$ such that
    $\ran T \subseteq \ran X$ and $TX \geq 0$;
  \item $T = AB$, $A,B \in \mc{L}^+$ and $A$ injective;
  \item $\sigma(T) \cap (-\infty,0) = \emptyset$, and there exists a
    quasi-affinity $X\in \mc{L}^+$ such that $TX = XT^*$ and
    $\ran T \subseteq \ran X$;
  \item There exists $C \in \mc{L}^+$ and a quasi-affinity
    $G \in L(\HH)$ such that $TG = GC$ and
    $\ran T \subseteq \ran(GG^*)$.
  \end{enumerate}
\end{thm}

\begin{proof}
  $(\mathit{i}\kern0.5pt) \Rightarrow (\mathit{ii}\kern0.5pt)$: This
  follows from Proposition~\ref{PropII}.

  $(\mathit{ii}\kern0.5pt) \Rightarrow (\mathit{iii}\kern0.5pt)$: Let
  $T = AB$, where $(A,B)$ is optimal.  Define
  $X := A+P_{\ker B} \in \mc{L}^+$.  Then
  $\ker X = \ker T^* \cap \clran T^* = (\ol{\ran T \dotplus \ker
    T})^\bot = \{0\}$, and so $X$ is a quasi-affinity.  Consequently,
  $TX = ABA \geq 0$ and $XB = AB = T$.  Hence
  $\ran T = \ran (XB) \subseteq \ran X$.

  $(\mathit{iii}\kern0.5pt) \Rightarrow (\mathit{iv}\kern0.5pt)$: Since
  $TX \geq 0$, $X$ is a quasi-affinity, and $\ran T \subseteq \ran X$,
  it follows from Douglas' lemma that $T = XP$ and $TX = XT^* = XPX
  \geq 0$, where $P \geq 0$.  So $T = XP \in \LL$, and by
  Lemma~\ref{lemaT1a}, $\sigma(T) \geq 0$.

  $(\mathit{iv}\kern0.5pt) \Rightarrow (\mathit{v}\kern0.5pt)$: If
  $T = AB$, $A,B \in \mc{L}^+$ and $A$ injective, then $X = A$ is a
  quasi-affinity and $TX \geq 0$.  By Lemma~\ref{lemaT1a},
  $\sigma(T) \geq 0$.

  $(\mathit{v}\kern0.5pt) \Rightarrow (\mathit{vi}\kern0.5pt)$: By
  Douglas' lemma, $T = XP$, and since $TX = XPX$ is selfadjoint and
  $X$ is a quasi-affinity, $P$ is selfadjoint.  By the assumption
  $\sigma(T) \cap (-\infty,0) = \emptyset$, it follows from
  \cite[Corollary~4.2]{MR2174236} that $TX \geq 0$, and hence that
  $P \geq 0$.  Thus $T\in \LL$.  Set $G = X^{1/2}$, which is also a
  quasi-affinity, and define $C = GPG \geq 0$.  Then $TG = GC$.  The
  last condition in $(\mathit{vi}\kern0.5pt)$ then follows since
  $T = XP$.
	
  $(\mathit{vi}\kern0.5pt) \Rightarrow (\mathit{i}\kern0.5pt)$: Since
  $T GG^* = GCG^* \geq 0$, and since $\ran T \subseteq \ran (GG^*)$,
  by Douglas' lemma $T = GG^*P$.  Moreover, $TG = G(G^*PG)$, and since
  $G$ is a quasi-affinity, $G^*PG = C$.  Hence $P \geq 0$ and so
  $T\in \LL$.
\end{proof}

\begin{obs}
  \label{obs:not-all-quasi-affs-work}
  A simple example shows that even if $T\in \LL$ and there is a
  quasi-affinity $X\in \mc{L}^+$ such that $TX = XT^* \geq 0$, it need
  not be true that $\ran T \subseteq \ran X$.  For example, take
  $T = 1$ on an infinite dimensional Hilbert space, and
  $X\in \mc{L}^+$, but without closed range.  Also, $T = C = 1$ and
  $G$ any quasi-affinity without closed range together satisfy
  $TG = GC$, but obviously, $\ran T$ is not contained in
  $\ran (GG^*)$.
\end{obs}

In~\cite[Corollary~3]{MR635584}, Stampfli showed that quasi-similar
operators with Dunford's property $C$ have equal spectra.  Since by
Theorem~\ref{prop:T-in-L+2-gen-scalar}, any $T \in \LL$ has property
$C$, and positive operators, being scalar, also have this property, it
follows that if $T\in \LL$ is quasi-similar to a positive operator
$C$, then $\sigma(T) = \sigma(C)$.  As the next result shows, this
continues to be true with the weaker assumption of quasi-affinity, and
as a bonus, the proof does not use any of the material from
Section~\ref{sec:l+2-spectral-properties}.

\begin{prop}
  \label{prop:t-qa-pos-then-spectra-equal}
  If $T\in \LL$ is quasi-affine to $C \in \mc{L}^+$, then
  $\sigma(T) = \sigma(C)$.
\end{prop}

\begin{proof}
  Suppose to begin with that $T$ is quasi-similar to $C$.  Write,
  using Theorem~\ref{PropI}, $T = AB$, $A,B \in \mc{L}^+$ and $A$
  injective.  Then $T$ is quasi-affine to $C_A := A^{1/2}BA^{1/2}$
  (with quasi-affinity $A^{1/2}$).  Applying
  Lemma~\ref{lem:T-qa-pos-C-then-sp_T-contains-sp_C}, $C$ is
  quasi-affine to $C_A$.  From \cite[Lemma~4.1]{MR250106}, $C$ and
  $C_A$ are unitarily equivalent, and so have equal spectra.  As
  noted in Lemma~\ref{lemaT1a}, $T$ and $C_A$ also have equal spectra,
  so the result follows in this case.

  Now suppose that $T$ is just quasi-affine to $C$.  If
  $\N = \clran T^*$, $\N$ is invariant for $T^*$, and
  $T^*P_\mc{N} \in \LL$ by
  Proposition~\ref{prop:L+2-restricted-to-inv-subsp-in-L+2}.  As
  observed in Lemma~\ref{lemaT1a},
  $\sigma(T) \subseteq {\mathbb R}^+$, so $\sigma(T^*) = \sigma(T)$,
  and consequently
  \begin{equation*}
    \sigma(T) = \sigma(T^*) = \sigma(T^*P_\mc{N}) = \sigma(P_\mc{N}
    T).
  \end{equation*}
  The middle equality follows by the same argument as in the proof of
  Lemma~\ref{lemaT1a}.

  Define $\tilde{T}: \mc{N} \to \mc{N}$ as the compression of $T$ to
  $\mc{N}$.  If $0 \notin\sigma(\tilde{T})$, so that $\tilde{T}$ is
  invertible, then $\ran T^* = \mc{N}$.  By Proposition~\ref{propCR1},
  $T$ is similar to a positive operator, and by
  Lemma~\ref{lem:Propqs}, $T$ is quasi-similar to $C$, and this has
  already been dealt with.

  If $0\in \sigma(\tilde{T})$, then
  $\sigma(P_\mc{N} T) = \sigma(\tilde{T})$.  Suppose that $TX = XC$,
  $X$ a quasi-affinity.  Then
  $\mc{R} := \overline{X^* \mc{N}} = \clran C$.  Since
  $X^*P_\N = P_{\mc{R}} X^* P_\N$, for
  $\tilde{C} = P_\mc{R} C |_\mc{R}$ and
  $\tilde{X} = P_\mc{N} X |_\mc{R}$, $\tilde{X}$ is a quasi-affinity
  and $\tilde{T}\tilde{X} = \tilde{X}\tilde{C}$.  Note that
  $\sigma(\tilde{C}) \cup\{0\} = \sigma(C) \cup \{0\}$, and if $0\in
  \sigma(C)$, then $0\in \sigma(\tilde{C})$.  By
  Theorem~\ref{PropI},
  $\overline{\ran\tilde{T} \dotplus \ker\tilde{T}} = \mc{N}$.  By
  definition,
  $\overline{\ran\tilde{T^*} \dotplus \ker\tilde{T^*}} =
  \overline{\ran\tilde{T^*} \dotplus \{0\}} = \mc{N}$.  Applying
  Theorem~\ref{PropI} to $\tilde{T}$ and ${\tilde{T}}^*$, $\tilde{T}$
  is quasi-similar to some positive operator, $C'$.  It then follows
  from Lemma~\ref{lem:Propqs} that $\tilde{T}$ is quasi-similar to
  $\tilde{C}$.  Hence $\sigma(\tilde{T}) = \sigma(\tilde{C})$.
  Finally,
  \begin{equation*}
    \sigma(\tilde{T}) = \sigma(T) \supseteq \sigma(C) =
    \sigma(\tilde{C}),
  \end{equation*}
  where the containment is by
  Lemma~\ref{lem:T-qa-pos-C-then-sp_T-contains-sp_C}, and the second
  equality follows since $0\in \sigma(\tilde{C})$.  Consequently,
  equality holds throughout.
\end{proof}

The next is a corollary of Theorem~\ref{PropI}.

\begin{cor}
  \label{propqs2}
  For $T \in L(\HH)$, the following are equivalent:
  \begin{enumerate}
  \item $T \in \LL$ and $T$ is quasi-similar to a positive operator;
  \item $T \in \LL$ and $\ol{\clran T\dotplus \ker T} = \HH$;
  \item There exist quasi-affinities $X, Y \in \mc{L}^+$ such that
    $\ran T \subseteq \ran X$, $\ran T^* \subseteq \ran Y$, $TX \geq
    0$, and $TY \geq 0$;
  \item $\sigma(T) \cap (-\infty,0) = \emptyset$, and there exists
    quasi-affinities $X,Y \in \mc{L}^+$ such that $TX = XT^*$,
    $YT = T^*Y$, and either $\ran T \subseteq \ran X$ or
    $\ran T^* \subseteq \ran Y$;
  \item There exists $C \in \mc{L}^+$ and quasi-affinities
    $G,F \in L(\HH)$ such that $TG = GC$, $FT = CF$, and either
    $\ran T \subseteq \ran (GG^*)$ or
    $\ran T^* \subseteq \ran (F^*F)$.
  \end{enumerate}
\end{cor}

\begin{proof}
  The equivalence of $(\mathit{i}\kern0.5pt)$ and
  $(\mathit{ii}\kern0.5pt)$ in Theorem~\ref{PropI} gives the
  equivalence of the first two items here.  Assuming $T \in \LL$ and
  $T$ quasi-similar to a positive operator, one has $T^*$ quasi-affine
  to a positive operator, and from this
  $\ol{\ran T^*\dotplus \ker T^*} = \HH$.  Taking orthogonal
  complements gives $\clran T \cap \ker T = \{0\}$ and so
  $\ol{\clran T\dotplus \ker T} = \HH$.  On the other hand, if
  $\ol{\clran T\dotplus \ker T} = \HH$, then
  $\clran T \cap \ker T = \{0\}$, and so taking orthogonal
  complements, $\ol{\clran T^*\dotplus \ker T^*} = \HH$.

  Consequently, Theorem~\ref{PropI} applies to both $T$ and $T^*$.
  Since $\sigma(T^*) = \{\overline{\lambda} : \lambda \in
  \sigma(T)\}$, $\sigma(T^*) \cap (-\infty,0) = \emptyset$ as well.
  The rest of the equivalences then easily follow.
\end{proof}

\begin{cor}
  \label{cor:L+2-w-A-or-B-closed-range}
  If $T \in \LL$ and $T = AB$ where $(A,B)$ is an optimal pair and
  either $\ran B$, respectively $\ran A$, is closed, then $T$,
  respectively $T^*$ is quasi-affine to a positive operator.  If there
  is such a pair with both $\ran A$ and $\ran B$ closed, then $T$ is
  quasi-similar to a positive operator.
\end{cor}

\begin{proof}
  Suppose $T \in \LL$ and $T = AB$ where $(A,B)$ is an optimal pair
  and $\ran B$ is closed.  From Proposition~\ref{lemaT2},
  $\ran B \cap \ker A = \{0\}$, and taking orthogonal complements
  gives that $\ker T \dotplus \ran T$ is dense in $\HH$.  Therefore,
  by Theorem~\ref{PropI}, $T$ is quasi-affine to a positive operator.
  The other case is handled identically.  If both $A$ and $B$ have
  closed range, $T$ and $T^*$ are both quasi-affine to positive
  operators.  By Lemma~\ref{lem:Propqs}, $T$ is quasi-similar to a
  positive operator.
\end{proof}

\begin{obs}
  \label{obs:LL-but-no-quasi-aff-to-pos}
  As was noted in Example~\ref{example-1} in
  Section~\ref{sec:the-set-l+2}, there exists an operator $T\in \LL$
  for which neither $\ran T \dotplus \ker T$ nor
  $\ran T^* \dotplus \ker T^*$ are dense.  Hence by the results of
  this section, in this particular example neither $T$ nor $T^*$ is
  quasi-affine to a positive operator, and in particular, $T$ will not
  be quasi-similar to a positive operator.

  On the other hand,
  Example~\ref{exmpl:qs-to-pos-but-no-A-or-B-w-clsd-rng} gives an
  operator $T \in \LL$ which is quasi-similar to a positive where
  there is no optimal pair $(A,B)$ with $\ran A$ or $\ran B$ closed,
  so there is no converse to
  Corollary~\ref{cor:L+2-w-A-or-B-closed-range}.
\end{obs}

While the operators which are similar to a positive operator are in
$\LL$, this is no longer necessarily true for those which are
quasi-similar to a positive operator.

\begin{prop}
  \label{prop:qs-pos-not-nec-in+L+2}
  For $T\in \LL$, $T$ is quasi-affine, respectively quasi-similar, to
  a positive operator if and only if $T$ has a square root which is
  quasi-affine, respectively, quasi-similar to a positive operator.
  Consequently, there exists an operator which is quasi-similar to a
  positive operator which is not in $\LL$.
\end{prop}

\begin{proof}
  Suppose that $T$ is quasi-affine to a positive operator.  By
  Theorem~\ref{PropI}, $T = AB$ with $A,B\in \mc{L}^+$ and $A$
  injective.  Set $C = A^{1/2}BA^{1/2}$ and $X = A^{1/2}$.  Then
  $TX = XC$.  By Douglas' lemma, $C^{1/2} = A^{1/2}B^{1/2}G$, so in
  particular, $\ran C^{1/2}X^* \subseteq \ran X^*$, and so another
  application of Douglas' lemma gives $R \in L(\HH)$ such that
  $RX = XC^{1/2}$.  Thus $R$ is quasi-affine to $C^{1/2}$.
  Since $R^2X = XC = TX$ and $\ran X$ is dense, $R^2 = T$.
  Conversely, if $R^2 = T$ and $RX = XD$, $D \geq 0 $ and $X$ a
  quasi-affinity, then $R^2 X = X D^2$.

  If in addition, $T$ is quasi-similar to $C$, $YR^2 = CY$.  Hence
  $C(YX) = YR^2X = (YX)C$, and by
  Lemma~\ref{lem:sq-root-of-pos-op-unique-in-L+2},
  $C^{1/2}(YX) = (YX)C^{1/2}$.  Therefore,
  \begin{equation*}
    YR(XC^{1/2}) = YR^2X = C(YX) = C^{1/2}Y (XC^{1/2}).
  \end{equation*}
  It is straightforward to see that $\clran (XC^{1/2}) = \clran T$,
  $\clran (R^*Y^*) = \clran T^*$, and
  $\clran (Y^*C^{1/2}) = \clran T^*$.  Also,
  $\ker (R^*Y^*) = \ker (Y^*C^{1/2}) = \ker T$.  By
  Corollary~\ref{propqs2},
  $\overline{\clran T \dotplus \ker T} = \HH$, and so $YR = C^{1/2}Y$
  on a dense subset.  By continuity, $YR = C^{1/2}Y$ on all of $\HH$,
  and thus $R$ is quasi-similar to $C^{1/2}$.  The converse is as in
  the quasi-affine case.

  By Example~\ref{ex:T-in-L+2-but-no-sq-rt-in-L+2}, there is a
  $T \in \LL$ which is quasi-similar to a positive operator, yet does
  not have a square root in $\LL$.  Hence for this operator, the
  square root constructed above is quasi-similar to a positive
  operator but is not in $\LL$.
\end{proof}

\begin{obs}
  \label{obs:sp-qa-pos-equals-sp-sq-rt-of-pos-in-ex}
  In the last proposition, the operator $R$ constructed there has
  $\sigma(R) \subseteq \sigma(C^{1/2}) \cup \sigma(-C^{1/2})$, which
  is a slight strengthening of
  Lemma~\ref{lem:T-qa-pos-C-then-sp_T-contains-sp_C} in this setting.

  There is no obvious way to rule out negative values in $\sigma(R)$
  if $R$ is not in $\LL$.  Nevertheless, for the operator in
  Example~\ref{ex:T-in-L+2-but-no-sq-rt-in-L+2}, the square root is
  $R = \begin{pmatrix} S^{1/2} & 0 \\ (1-S)^{1/2} & 0 \end{pmatrix}$,
  which happens to be a partial isometry with
  $\sigma (R) = \sigma(S^{1/2}) \cup \{0\} = \sigma(S^{1/2})$.  For
  \begin{equation*}
    C = \begin{pmatrix} S & 0 \\ 0 & 0 \end{pmatrix}, \quad
    X = \begin{pmatrix} S^{1/2} & 0 \\ (1-S)^{1/2} & 1 \end{pmatrix},
    \quad
    Y = \begin{pmatrix} 1 & 0 \\ -(1-S)^{1/2} & S^{1/2} \end{pmatrix},
  \end{equation*}
  $C \geq 0$, $X$ and $Y$ are quasi-affinities, $RX = XC^{1/2}$ and
  $YR = C^{1/2}Y$.  So even though $R \notin \LL$,
  $\sigma(R) = \sigma(C^{1/2})$.
\end{obs}

\section{$\LL$ -- the general case}
\label{sec:l+2--general-case}

The sole remaining case to consider are those operators $T \in \LL$
for which neither $\M := \ol{\ran T \dotplus \ker T}$ nor
$\N := \ol{\ran T^* \dotplus \ker T^*}$ equals $\HH$.  Decompose
\begin{equation*}
  \HH= \M \oplus (\clran T^* \cap \ker T^*) = \N \oplus (\clran T \cap
  \ker T).
\end{equation*}
The spaces $\M$ and $\N^\bot$ are invariant for $T$, while $\N$ and
$\M^\bot$ are invariant for $T^*$.  In what follows, statements
involving only the spaces $\M$ and $\M^\bot$ are given, since it is
obvious what the equivalent statements for $\N$ and $\N^\bot$ should
be.

\begin{lema}
  \label{lem:T_M-qa-to-pos-op}
  Let $T\in \LL$.  Then $T_\M := T P_\M \in \LL$,
  $\clran T_\M = \clran T$, $\ker T_\M = \ker T \oplus \M^\bot$, and
  $T_\M$ is quasi-affine to a positive operator.  Also, if $(A,B)$ is
  an optimal pair for $T$, then
  $\ran T \subseteq \ran (A(P_\M B P_\M)^{1/2})$ and $(A,P_\M B P_\M)$
  is an optimal pair for $T_\M$.
\end{lema}

\begin{proof}
  Applying Proposition~\ref{prop:L+2-restricted-to-inv-subsp-in-L+2}
  and Corollary~\ref{cor:range-maps-densely}, $T_\M \in \LL$ and
  $\clran T_\M = \clran T$.

  Write $T = AB$, where $(A,B)$ is optimal.  Since
  $\ker B = \ker T \subseteq \M$ and $P_\M A = A = A P_\M$,
  \begin{equation*}
    T_\M = (A + P_{\ker T} + P_{\M^\bot})(P_\M B P_\M),
  \end{equation*}
  where $A + P_{\ker T} + P_{\M^\bot}$ is positive and injective since
  by now standard calculations, $A + P_{\ker T}$ has this property on
  $\M$.  It then follows from Theorem~\ref{PropI} that $T_\M$ is
  quasi-affine to a positive operator.  Also
  $\ker T_\M = \ker(P_\M B P_\M) = \ker T \oplus \M^\bot$.

  Finally, since $B \geq 0$,
  $\ran(P_\M B P_{\M^\bot}) \subseteq \ran (P_\M B P_\M)^{1/2}$.  Then
  from
  \begin{equation*}
    T =
    \begin{pmatrix}
      T_\M & TP_{\M^\bot}
    \end{pmatrix}
    = A
    \begin{pmatrix}
      P_\M B P_\M & P_\M B P_{\M^\bot}
    \end{pmatrix},
  \end{equation*}
  the last claim follows.
\end{proof}

It is also true that $T$ is the restriction of an operator in $\LL$
which is quasi-affine to a positive operator in the following sense.

\begin{lema}
  \label{lem:T_M-the-restriction-of-an-op-qa-to-a-pos-op}
  Let $T\in \LL$.  Then there is an operator $T^\M \in \LL$ with the
  properties that $T^\M$ is quasi-affine to a positive operator,
  $T = P_\M T^\M$, $\clran T^\M = \clran T \oplus \M^\bot$ and
  $\ker T^\M = \ker T$.
\end{lema}

\begin{proof}
  Write $T = AB$ with $(A,B)$ optimal.  Set
  \begin{equation*}
    T^\M = T + P_{\M^\bot}B = (A+P_{\M^\bot}) B = (A + P_{\ker T} +
    P_{\M^\bot})B.
  \end{equation*}
  Then $T^\M \in \LL$, $A + P_{\ker T} + P_{\M^\bot} \geq 0$ is
  injective, and $\ker T^\M = \ker B$.  By Theorem~\ref{PropI}, $T^\M$
  is quasi-affine to a positive operator.

  Since $T^{\M\,*} = B(A + P_{\M^\bot})$,
  $\ker T^{\M\,*} \supseteq \ker(A+P_{\M^\bot}) = \ker A \cap \M$, and
  if $T^{\M\,*}x = 0$, then
  $(A+P_{\M^\bot})x \in \ker B \cap \ran(A+P_{\M^\bot}) = \{0\}$.  The
  last equality follows since if $x\in \ker B \subseteq \mc{M}$ and
  $x = x_1+x_2$, $x_1 \in \ran A \subseteq \mc{M}$ and
  $x_2 \in \mc{ M}^\bot$, then $x_2 = 0$, and since
  $\ran A \cap \ker B = \{0\}$ by Proposition~\ref{lemaT2}, $x = 0$.
  Hence $\ker T^{\M\,*} = \ker(A+P_{\M^\bot})$, and so
  $\clran T^\M = \clran T \oplus \M^\bot$.
\end{proof}

\begin{thm}
  \label{thm:general-case}
  Let $T \in L(\HH)$ and $\M = \overline{\ran T + \ker T}$.  The
  following are equivalent:
  \begin{enumerate}
  \item $T \in \LL$;
  \item $T_\M := TP_\M \in \LL$ and there exists an optimal pair
    $(A,B)$ for $T_\M$ such that $\ran T \subseteq AB^{1/2}$;
  \item There exists $T^\M\in \LL$ satisfying $T = P_\M T^\M$ and an
    optimal pair $(A,B)$ for $T^\M$ such that $A\M^\bot = \M^\bot$.
  \end{enumerate}
  In this case, both $T_\M$ and $T^\M$ are quasi-affine to positive
  operators.
\end{thm}

\begin{proof}
  $(\mathit{i}\kern0.5pt) \Rightarrow (\mathit{ii}\kern0.5pt)$ and
  $(\mathit{i}\kern0.5pt) \Rightarrow (\mathit{iii}\kern0.5pt)$ follow
  from the last two lemmas.

  Assume $(\mathit{ii}\kern0.5pt)$ holds and that $T_\M = AB$ for an
  optimal pair $(A,B)$ such that
  $\ran (TP_{\M^\bot}) \subseteq \ran T \subseteq \ran(AB^{1/2})$.  By
  Douglas' lemma, there is an operator $Z \in L(\HH)$ with
  $\ker Z = \ker (TP_{\M^\bot}) = \M$ and such that
  $TP_{\M^\bot} = AB^{1/2}Z$.  Hence,
  \begin{equation*}
    T = T_\M + TP_{\M^\bot} = A(B + B^{1/2}Z) = AB^{1/2}(B^{1/2} + Z) =
    A(B^{1/2} + Z^*)(B^{1/2} + Z)
  \end{equation*}
  is in $\LL$, the last equality following since
  $\clran Z^* \subseteq \M^\bot \subseteq \ker A$.  Thus
  $(\mathit{i}\kern0.5pt)$ holds.

  Now assume $(\mathit{iii}\kern0.5pt)$ is true.  Then for the optimal
  pair $(A,B)$ there,
  $P_\M A = P_\M (A P_\M + AP_{\M^\bot}) = P_\M A P_\M$.  Hence
  $T = P_\M A P_\M B \in \LL$, which is $(\mathit{i}\kern0.5pt)$.
\end{proof}

\begin{obs}
  \label{obs:-spec-ext-restr}
  Since $T_{\mc{M}} = T P_{\mc{M}}$,
  $\sigma(T_{\mc{M}}) \cup \{0\} = \sigma(P_{\mc{M}}T) \cup \{0\} =
  \sigma(T) \cup \{0\}$.  If $0 \notin \sigma(T)$, $P_{\mc{M}} = 1$,
  and likewise, if $0 \notin \sigma(T_{\mc{M}})$,
  $\ran P_{\mc{M}} = \HH$, so again $P_{\mc{M}} = 1$.  Thus,
  $\sigma(T_{\mc{M}}) = \sigma(T)$.  Unfortunately, there does not
  seem to be any similar relation between $\sigma(T^{\mc{M}})$ and
  $\sigma(T)$.

  There is also the dilation result for the class $\LL$ in
  Proposition~\ref{prop:dilations-for-L+2}, though there does not
  appear to be such a close connection for the spectra of these with
  that of $T$.  These dilations are in a sense extremal for the family
  $\LL$, in that any further dilations are direct sums.

  Theorem~\ref{prop:T-in-L+2-gen-scalar} indicates that all operators
  in $\LL$ are generalized scalar, so it is natural to wonder if there
  is some characterization of the class $\LL$ in terms of this
  property and the spectrum of the operator being in $\mathbb R^+$.
\end{obs}

\section{Examples}
\label{sec:examples}

Recall that an operator $T$ is \emph{algebraic} if there is a
polynomial $p$ such that $p(T) = 0$.  By the spectral mapping theorem,
$\sigma(T)$ is then contained in the set of roots of the polynomial.

\begin{prop}
  \label{prop:L+2-and-algebraic}
  Suppose that $T \in \LL$ is algebraic.  Then $T$ has the form
  \begin{equation*}
    T = \sum_j \lambda_j Q_j,
  \end{equation*}
  where each $\lambda_j \geq 0$ is an eigenvalue for $T$ and $Q_j$ is
  an oblique projection.  In this case, $\ran T$ is closed and $T$ is
  similar to $C = \sum_j \lambda_j P_j \geq 0$, where each $P_j$ is an
  orthogonal projection and $\bigoplus_j P_j = 1$.  Conversely, if $T$
  has this form, then $T \in \LL$ and is algebraic.
\end{prop}

\begin{proof}
  As noted above, if $p(T) = 0$ for a polynomial $p$, the spectrum of
  $T$ is a finite set of points taken from the non-negative roots of
  $p$.  For each $\lambda_j \in \sigma(T)$, let $Q_j$ be the Riesz
  projection for $\lambda_j \in \sigma(T)$.  Then $\HH_j = \ran Q_j$
  is invariant for $T$ and $\sigma(T|_{\HH_j}) = \lambda_j$.
  Furthermore, $Q_iQ_j = 0$ for $i\neq j$.  By
  Proposition~\ref{prop:L+2-restricted-to-inv-subsp-in-L+2} and
  Proposition~\ref{prop:X-sub-T-of-lambda-closed},
  $T|_{\HH_j} = GAG^{-1}$, with $A\geq 0$, $G$ invertible in $\HH_j$,
  and $\sigma(A) = \{\lambda_j\}$.  Thus $A = \lambda_j 1_{\HH_j}$,
  and so $T|_{\HH_j} = \lambda_j 1_{\HH_j}$.  Therefore,
  $TQ_j = \lambda_j Q_j$ for some oblique projection $Q_j$ and if
  $\lambda_j \neq 0$, $T|_{\HH_j}$ is invertible.  Since $Q_iQ_j = 0$
  when $i\neq j$, $\sum_j Q_j$ is a projection, and moreover
  $\sum_j Q_j = 1$.  So $T$ has the claimed form.

  Since $Q_iQ_j = 0$ when $i\neq j$, $\ran Q_i \dotplus \ran Q_j$ is
  closed, and consequently, $\ran T = \bigvee_{j\neq 0} \HH_j$ is
  closed.  So by Corollary~\ref{corCR}, $T$ is similar to a positive
  operator $C$.  In this case, $C$ must be as in the statement of the
  proposition.

  For the converse, the statement that $T$ is algebraic follows from
  the spectral mapping theorem, using a polynomial with roots equal to
  the set of eigenvalues.  Furthermore, $T$ is a scalar operator and
  its spectrum is in a set of the form $\{0\}\cup [c,\infty)$, $c >
  0$.  Therefore by Corollary~\ref{corCR2}, $T$ is in $\LL$.
\end{proof}

\begin{obs}
  \label{obs:optimal-pair-for-alg-op-in-L+2}
  Using Example~\ref{exmpl:oblique-proj-opt-pair}, it is possible to
  write down an optimal pair for any $T$ in $\LL$ which is algebraic.
  Let $(A_j,B_j)$ be the optimal pair for the oblique projection
  $Q_j$, as constructed in that example.  Claim that for
  $A = \sum_j A_j$, $B = \sum_j B_j$, $(A,B)$ is an optimal pair for
  $T$.  Since $Q_j Q_k = 0$ if $k \neq j$, $B_j A_k = 0$, or
  equivalently, $A_kB_j = 0$.  Hence $T = AB$.  It is immediate that
  $\ran A = \ran T$ and $\ran B = \ran T^*$, so the pair is optimal.
\end{obs}

Next consider the class of compact operators in $\LL$.

\begin{cor}
  \label{cor:compacts-in-L+2}
  Let $T$ be a compact operator in $\LL$ and let
  $\sigma(T) = \{\lambda_j\}$.  Then restricted to the range $\HH_j$
  of the Riesz projection corresponding to $\lambda_j \neq 0$,
  $T|_{\HH_j} = \lambda_j 1_{\HH_j}$.  Furthermore, $T$ has no
  quasi-nilpotent part other than $\ker T$.
\end{cor}

\begin{proof}
  The first part is obtained in the same way as in the proof of the
  last proposition,  The last part follows directly from
  Proposition~\ref{prop:X-sub-T-of-lambda-closed}.
\end{proof}

\begin{obs}
  \label{obs:bad-compacts-in-L+2}
  Despite the simple form of the eigenspaces for a compact operator in
  $\LL$, a compact operator is generally not as nice as an algebraic
  operator.  Indeed, it need not be quasi-affine to a positive
  operator, even if it is in a Schatten class.  It suffices to verify
  this with the trace class operators.

  For example, let $(e_n)_{n=1}^\infty$ be an orthonormal basis on
  $\HH$, and $\{\lambda_j\} \subset \mathbb R^+$ non-zero and
  absolutely summable.  Also let $P_n$ be the orthonormal projection
  onto the span of $e_n$.  Define $A = \sum_n \lambda_n P_n$, a
  positive trace class operator.  If $x = \sum_n \lambda_n e_n$, then
  $x\in \HH$, and there is obviously no vector $y\in \HH$ such that
  $Ay = x$.  As in Example~\ref{example-1}, define $B$ to be the
  orthogonal projection onto $(\bigvee x)^\bot$.  Then $T = AB$ is
  trace class, and is not quasi-similar to a positive operator.  With
  minor modifications, $T$ can be chosen to be trace class and not
  even quasi-affine to a positive operator.

  It follows from Apostol's theorem (Theorem~\ref{Apostol}) that the
  eigenspaces of $T$ do not form a basic system of subspaces.
  Moreover, it is not clear that a compact operator with eigenvalues
  and eigenspaces as in Corollary~\ref{cor:compacts-in-L+2} will
  necessarily be in $\LL$, even if the eigenspaces do form a basic
  system.
\end{obs}

Suppose that $T$ is compact and of the form given in
Corollary~\ref{cor:compacts-in-L+2}.  Suppose furthermore that
$\overline{\bigvee_n} \HH_n = \HH$.  In this case the eigenspaces form
a basic system.  For
$\{\alpha_n\} \subset \mathbb R^+ \backslash \{0\}$ with
$\sum_n \alpha _n < \infty$, define $X:\bigoplus_n \HH_n \to \HH$ by
\begin{equation*}
  X(\oplus_n x_n) = \sum_n \alpha_n x_n, \qquad x_n \in \HH_n.
\end{equation*}
Notice that $\bigoplus_n \HH_n$ is a sort of ``straightened'' version
of $\HH$ and is isomorphic to $\HH$.  By the arguments in
\cite{MR402522}, $X$ is bounded.  Let $Q_n:\bigoplus_n \HH_n \to \HH$
be the oblique projection defined by
\begin{equation*}
  Q_n x =
  \begin{cases}
    x, & x\in \HH_n; \\
    0, & x\in \bigoplus_{k \neq n} \HH_n.
  \end{cases}
\end{equation*}

\begin{lema}
  \label{lem:q-aff-adjoint}
  For $X$ defined as above,
  \begin{equation*}
    X^*y = \sum_n \alpha_n Q_n^* y, \qquad y \in \HH.
  \end{equation*}
\end{lema}

\begin{proof}
  As defined, $Q_n$ has the properties that $\clran Q_n^* = (\ker
  Q_n)^\bot = \HH_n$ and $\ker Q_n^* = (\ran Q_n)^\bot = \HH_n^\bot$.
  Thus, for $y \in \HH$ and $x = \oplus_n x_n \in \bigoplus_n \HH_n$,
  \begin{equation*}
    \begin{split}
      \PI{\sum_n \alpha_n Q_n^* y}{x}  &= \sum_n \alpha_n \PI{y}{Q_n
        x} = \sum_n \alpha_n \PI{y}{x_n} \\
      &= \PI{y}{\sum_n \alpha_n x_n} = \PI{y}{Xx}.\qedhere
    \end{split}      
  \end{equation*}
\end{proof}

\begin{prop}
  \label{prop:conds-st-T-cpt-in-L+2}
  Let $T$ be a compact operator in $L(\HH)$ with
  $\sigma(T) = \{\lambda_j\} \geq 0$, and suppose that when restricted
  to the the range $\HH_j$ of the Riesz projection corresponding to
  $\lambda_j \neq 0$, $T|_{\HH_j} = \lambda_j 1_{\HH_j}$, and that the
  quasi-nilpotent part of $T$ is the kernel.  If
  $\sum_j \lambda_j^{1/2} < \infty$, then $T \in \LL$.
\end{prop}

\begin{proof}
  Take $X$ defined as above, but with $\alpha_n = \lambda_n^{1/2}$ if
  $\lambda_n > 0$ and $1$ otherwise.  Then $\sum_n \alpha_n < \infty$
  and by Lemma~\ref{lem:q-aff-adjoint}, for
  $x = \oplus_n x_n \in \bigoplus_n \HH_n$,
  \begin{equation*}
    \begin{split}
      \|Xx\|^2 & = \PI{X^*(\sum_n \alpha_n x_n)}{\oplus_n x_n}
      = \sum_n \alpha_n \PI{X^*x_n}{\oplus_n x_n} \\
      & = \sum_n \alpha_n^2 \PI{\oplus_n x_n}{\oplus_n x_n}
      = \sum_n \alpha_n^2 \|x_n\|^2.
    \end{split}
  \end{equation*}
  Define $C \geq 0$ on $ \bigoplus_n \HH_n$ by $\PI{Cx}{x} = \sum_n
  \lambda_n \|x_n\|^2$.  Then $X$ is a quasi-affinity and $X^*X \geq
  C$ (in fact, it will be equal if $\ker T = \{0\}$).  By Douglas'
  lemma, $C^{1/2} = X^*Z$ for bounded $Z$.  By Apostol's theorem (or
  rather, the proof of it),
  \begin{equation*}
    TX = XC = XX^*ZZ^*X,
  \end{equation*}
  and so since $\ran X$ is dense, $T \in \LL$.
\end{proof}

Next consider Fredholm operators in $\LL$.  Recall that $T$ is
\emph{left-semi-Fredholm} if there exists a bounded operator $R$ and a
compact operator $K$ such that $RT = 1+K$.  On the other hand, it is
\emph{right semi-Fredholm} if there exist such $R$ and $K$ such that
$TR = 1 + K$.  Finally, $T$ is \emph{Fredholm} if it is both left and
right semi-Fredholm.

\begin{prop}
  \label{prop:L+2-and-Fredholm}
  Let $T \in \LL$.  Then $T$ is left / right semi-Fredholm if and only
  if $T$ is Fredholm and similar to a positive operator with closed
  range and finite dimensional kernel.  In this case,
  \begin{equation*}
    \mathrm{ind}\, T := \dim \ker T - \dim \ker T^* = 0.
  \end{equation*}
\end{prop}

\begin{proof}
  Suppose that $T \in \LL$ and that it is left semi-Fredholm (the
  other case is handled identically).  Then by Atkinson's theorem,
  $\ran T$ is closed and $\dim \ker T < \infty$.  Hence by
  Proposition~\ref{propCR1}, $T$ is similar to a positive operator.
  If $T = LCL^{-1}$ where $C \geq 0$ and $L$ is invertible, $\ran T$
  closed implies that $\ran C$ is closed, and $\dim \ker T < \infty$
  gives $\dim \ker C < \infty$.  Since $T^* = L^{*\,-1} C L^*$,
  $\dim \ker T^* = \dim \ker T$.
\end{proof}

\section*{Acknowledgments}

Maximiliano Contino was supported by the UBA's Strategic Research Fund
2018 and CONICET PIP 0168.  Alejandra Maestripieri was supported by
CONICET PIP 0168.  The work of Stefania Marcantognini was done during
her stay at the Instituto Argentino de Matem\'atica with an
appointment funded by the CONICET.  She is very grateful to the
institute for its hospitality and to the CONICET for financing her
post.

  \goodbreak

\section*{Data availability statement}

No data were produced for this paper.

\end{document}